\documentclass[11pt,reqno]{amsart}

\usepackage{xcolor,soul}
\usepackage{amsmath,amsthm,amssymb,tocvsec2,mathrsfs}
\usepackage[tmargin=1.4in,bmargin=1.2in,rmargin=1.4in,lmargin=1.4in]{geometry}
\usepackage[breaklinks=true]{hyperref}

\theoremstyle{plain}

\newtheorem{theorem}{Theorem}[section]
\newtheorem{corollary}[theorem]{Corollary}
\newtheorem{proposition}[theorem]{Proposition}
\newtheorem{lemma}[theorem]{Lemma}

\newtheorem{utheorem}{\textrm{\textbf{Theorem}}}

\theoremstyle{definition}

\newtheorem{definition}[theorem]{Definition}
\newtheorem{notation}[theorem]{Notation}

\newtheorem{remark}[theorem]{Remark}

\numberwithin{equation}{section}
\numberwithin{table}{section}

\newcommand{\adj}{\mathop{\mathrm{adj}}}
\newcommand{\ba}{\mathbf{a}}
\newcommand{\bA}{\mathbf{A}}
\newcommand{\balpha}{{\boldsymbol{\alpha}}}
\newcommand{\bbeta}{{\boldsymbol{\beta}}}

\newcommand{\bB}{\mathbf{B}}

\newcommand{\bC}{\mathbf{C}}
\newcommand{\be}{\mathbf{e}}
\newcommand{\bepsilon}{{\boldsymbol{\epsilon}}}
\newcommand{\bk}{\mathbf{k}}
\newcommand{\bm}{\mathbf{m}}
\newcommand{\bs}{\mathbf{s}}
\newcommand{\bu}{\mathbf{u}}
\newcommand{\bv}{\mathbf{v}}
\newcommand{\bw}{\mathbf{w}}
\newcommand{\bx}{\mathbf{x}}
\newcommand{\bz}{\mathbf{z}}
\newcommand{\bone}{\mathbf{1}}
\newcommand{\bzero}{\mathbf{0}}

\newcommand{\cS}{\mathcal{S}}
\newcommand{\cW}{\mathcal{W}}
\newcommand{\diag}{\mathop{\mathrm{diag}}}
\newcommand{\fin}{\mathrm{fin}}

\newcommand{\Id}{\mathrm{Id}}
\newcommand{\rank}{\mathop{\mathrm{rank}}}

\newcommand{\sP}{\mathscr{P}}

\newcommand{\symdiff}{\bigtriangleup}

\newcommand{\up}[1][]{\Sigma^\uparrow_{\pi_{#1}}}
\newcommand{\wt}[1]{\widetilde{#1}}

\newcommand{\C}{\mathbb{C}}

\newcommand{\R}{\mathbb{R}}
\newcommand{\Z}{\mathbb{Z}}

\newcommand{\tup}[1]{\textup{#1}}
\newcommand{\eps}{t_0}

\renewcommand{\subset}{\subseteq}

\allowdisplaybreaks

\begin{document}
\title[Negativity-preserving transforms of tuples of symmetric
matrices]{Negativity-preserving transforms of\\tuples of symmetric matrices}

\author{Alexander Belton}
\address[A.~Belton]{School of Engineering, Computing and Mathematics,
University of Plymouth, UK}
\email{\tt alexander.belton@plymouth.ac.uk}

\author{Dominique Guillot}
\address[D.~Guillot]{University of Delaware, Newark, DE, USA}
\email{\tt dguillot@udel.edu}

\author{Apoorva Khare}
\address[A.~Khare]{Department of Mathematics, Indian Institute of
Science, Bangalore, India and Analysis \& Probability Research Group,
Bangalore, India}
\email{\tt khare@iisc.ac.in}

\author{Mihai Putinar}
\address[M.~Putinar]{University of California at Santa Barbara, CA, USA} 
\email{\tt mputinar@math.ucsb.edu}

\date{27th February 2026}

\begin{abstract}
Compared to the entrywise transforms which preserve positive
semidefiniteness, those leaving invariant the inertia of symmetric
matrices reveal a surprising rigidity. We first obtain the classification
of negativity preservers by combining recent advances in matrix analysis
with some novel arguments relying on well chosen test matrices,
Sidon sets from number theory, and analytic properties of
absolutely monotone functions. We continue with the analogous
classification in the multi-variable setting, revealing for the
first time a striking separation of variables, with absolute
monotonicity on one side and only homotheties on the other. We
conclude with the complex analogue of this result.
\end{abstract}



\maketitle

\settocdepth{section}
\tableofcontents

\section{Introduction and main results}\label{S1}

A subtitle of this article could be ``\emph{Indefinite variations on a
theme by Schoenberg}.'' We honor here Schoenberg the mathematician, whose
legacy is arguably comparable to that of his homonym, the artist.
As much as Arnold Schoenberg's musical innovation has shaped the last
century, the same applies to Isaac Schoenberg in the mathematical milieu.
We touch upon below only a concise, but profound, 1942 contribution of
the latter~\cite{Schoenberg42}: a definitive description of positive
definite functions defined on Euclidean spheres. This theorem synthesizes
several decades of metric geometry, a topic much cultivated by Schoenberg
(partially in the company of von Neumann). A second motivation comes from
his early studies of total positivity, in themselves rooted in the
analysis of the oscillatory behavior of analytic functions or matrices.
From this path of discovery the concept of splines would later emerge. A
reformulation of Schoenberg's classification of positive definite
functions on spheres turned out to be a guiding light for generations to
come. We state this theorem in a form which includes some slight
enhancements accumulated over time.

Henceforth, a function $f : I \to \R$ acts entrywise on a matrix
$A = (a_{ij})$ with entries in $I$ via the prescription
$f[A] := \bigl( f( a_{i j} ) \bigr)$.

\begin{theorem}[\cite{Schoenberg42, Rudin59, BGKP-hankel}]\label{Tschoenberg}
Let $I := (-\rho,\rho)$, where $0 < \rho \leq \infty$.
Given a function $f : I \to \R$, the following are equivalent.
\begin{enumerate}
\item The function $f$ acts entrywise to preserve the set of
positive semidefinite matrices of all dimensions with entries in~$I$.
\item The function $f$ is \emph{absolutely monotone}, that is,
$f(x) = \sum_{n=0}^\infty c_n x^n$ for all $x \in I$
with $c_n \geq 0$ for all $n$.
\end{enumerate}
\end{theorem}

We use the term \emph{absolutely monotone} to describe
functions which have a power-series representation with non-negative
Maclaurin coefficients, although the non-negativity of derivatives holds
only for a certain subset of the domain (which is~$[ 0, \rho )$ above).
For more on this, see Appendix~\ref{Appendix}.

Entrywise transforms that preserve positive definiteness have been
investigated within the context of linear algebra or harmonic analysis by
many authors: Rudin \cite{Rudin59}, Herz \cite{Herz63}, Loewner and
Horn \cite{horn}, Christensen and Ressel \cite{ChrRes1,ChrRes2} (who also
have a relevant monograph with Berg \cite{BCR-book}), Vasudeva \cite{vasudeva79},
and FitzGerald, Micchelli and Pinkus \cite{fitzgerald}, to name just some of those
whose work in this area appeared in the second half of the 20th century. More recently,
we have articles by Fallat, Johnson and Sokal
\cite{FJS}, Jain \cite{Jain,Jain2}, and Vishwakarma \cite{Vish}, in
addition to our previous publications \cite{BGKP-fixeddim,BGKP-hankel}
and those with additional co-authors\ \cite{GKR-lowrank,KT}). We note also that
Moln\'ar \cite{Molnar} has studied similar transforms in the context of
operator algebras.

Interest in this subject was reinvigorated by statisticians.
Their quest was prompted by the implementation of thresholding or other
entrywise operations aimed at regularizing large correlation matrices
which are near, but not, sparse. Details about this line of enquiry
appear in \cite{GuillotRajaratnam2012}, \cite[Section 5]{BGKP-survey2},
and the monograph \cite{Khare}.

The present article is a part of a systematic study and classification of
inner transformations of structured matrices and kernels. Our earlier
works dealt with entrywise positivity preservers: see, in particular,
\cite{BGKP-hankel} and \cite{BGKP-TN}. We now turn to indefinite
quadratic forms. In addition to the unquestionable theoretical interest,
a strong motivation to pursue this investigation comes from the recently
uncovered benefits of embedding data, and particularly big data, in
hyperbolic space. This is a lively topic in full spate today, surfacing
in an array of areas such as image and language processing
\cite{DGRS,Kh-IEEE}, finance \cite{KRN}, social networks \cite{Verbeek},
and geographic routing \cite{Kleinberg}. We do not explore such
applications in the present work, focusing solely on describing the
entrywise preservers of matrices with negative-eigenvalue constraints.
The surprising rigidity revealed by the classification completed
below would suggest profound impacts on data structuring. We will
elaborate on such applications in a separate article.

We provide herein a comprehensive extension to previous Schoenberg-type 
theorems, old and new, such as  \cite{BGKP-hankel}. To be more specific,
Theorem~\ref{Tschoenberg} classifies all entrywise inner transforms of
self-adjoint matrices with no negative eigenvalues, tacitly allowing 
the nullity or number of positive eigenvalues to vary.
(Throughout this work, eigenvalues are counted with multiplicity.)
Instead, we now seek operations which do not change the inertia of
real symmetric matrices. More generally, we provide complete
answers to three questions.
\begin{itemize}
\item What are the entrywise transforms sending matrices with at most $k$
negative eigenvalues to ones with at most $l$ negative eigenvalues, for
arbitrary non-negative integers $k$ and $l$?

\item What are the preservers of inertia on such sets of matrices?
(We recall that the \emph{inertia} of a complex
Hermitian matrix $A$ is the triple $( n_+, n_0, n_- )$ of non-negative integers
corresponding to the number of positive, zero and negative eigenvalues of $A$.
Some authors prefer the term \emph{signature} for this triple, but as this is also used to
describe the difference $n_+ - n_-$ we will prefer the former terminology.)
\item What are the multi-variable analogues of these questions, in both
the real and complex settings?
\end{itemize}
The classes of transforms that answer these questions turn out to be far
smaller than the collection of functions appearing in
Theorem~\ref{Tschoenberg}.

\subsection{One-variable inertia preservers.}
A step further from Schoenberg's Theorem is the description of 
entrywise transforms that preserve the inertia of matrices with precisely
$k$ negative eigenvalues for some choice of integer $k$. To state our
first complete result precisely, we introduce the following notation.

Given non-negative integers $n$ and $k$, with $n \geq 1$ and $k \leq n$,
we let~$\cS_n^{(k)}(I)$ denote the set of $n \times n$ symmetric matrices
with entries in $I \subseteq \R$ having exactly $k$ negative eigenvalues;
here and throughout, eigenvalues are counted with multiplicity. Let
\[
\cS^{(k)}(I) := \bigcup_{n = k}^\infty \cS_n^{(k)}(I)
\]
be the set of real symmetric matrices of arbitrary size with
entries in $I$ and exactly $k$ negative eigenvalues.
For brevity we let $\cS^{(k)} := \cS^{(k)}(\R)$ and
$\cS_n^{(k)} := \cS_n^{(k)}(\R)$.

Note that, for any $n \geq 1$, the sets
$\cS_n^{(0)}$, $\cS_n^{(1)}$, $\ldots,$ $\cS_n^{(n)}$ are pairwise
disjoint and partition the set of $n \times n$ real symmetric matrices.

We now assert

\begin{theorem}\label{Tinertia}
Let $I := (-\rho,\rho)$, where $0 < \rho \leq \infty$, and let $k$ be a
non-negative integer.
Given a function $f : I \to \R$, the following are equivalent.
\begin{enumerate}
\item The entrywise transform $f[-]$ preserves the inertia of all
matrices in $\cS^{(k)}(I)$.

\item The function is a positive homothety: $f(x) \equiv c x$ for some
constant $c > 0$.
\end{enumerate}
\end{theorem}

Thus, the class of inertia preservers for the collection of real symmetric
matrices of all sizes with $k$ negative eigenvalues is highly restricted,
whatever the choice of $k$: every such map in fact preserves not only the
nullity and the total multiplicities of positive and negative
eigenvalues, it preserves the eigenvalues themselves, up to simultaneous
scaling.

Our second result resolves the dimension-free preserver problem for
$\cS^{(k)}(I)$. If $k=0$, Schoenberg's Theorem~\ref{Tschoenberg} shows
that the class of entrywise preservers is far larger than the class of
inertia preservers, which contains only the positive homotheties.
However, if $k>0$ then this is no longer the case.

\begin{theorem}\label{Tpreserver}
Let $I := (-\rho,\rho)$, where $0 < \rho \leq \infty$, and let $k$ be a
positive integer.
Given a function $f : I \to \R$, the following are equivalent.
\begin{enumerate}
\item The entrywise transform $f[-]$ sends $\cS^{(k)}(I)$ to $\cS^{(k)}$.

\item The function $f$ is a positive homothety, so that
$f( x ) \equiv c x$ for some $c > 0$, or, when $k = 1$, we can also have
that $f( x ) \equiv -c$ for some $c > 0$.
\end{enumerate}
\end{theorem}

There is a notable rigidity phenomenon here, in stark contrast to the
dimension-free preserver problem for positive semidefinite matrices (the
$k = 0$ case).
When there is at least one negative eigenvalue, the non-constant
transforms leaving invariant the number of negative eigenvalues also
conserve the number of positive eigenvalues and the number of zero
eigenvalues; more strongly, they preserve the eigenvalues themselves, up
to simultaneous scaling.
That is, Schoenberg's theorem collapses to just homotheties if
$k \geq 2$, with the additional appearance of negative constant functions
if $k = 1$ (and the collection of preservers is non-convex in this last
case).

It is interesting to compare these results with a theorem obtained about
three decades ago by FitzGerald, Micchelli, and Pinkus \cite{fitzgerald},
who classified the entrywise preservers of conditionally positive
matrices of all sizes. An $n \times n$ real symmetric matrix $A$ is
\emph{conditionally positive} if the corresponding quadratic form is
positive semidefinite when restricted to the hyperplane
$\bone_n^\perp \subset \R^n$, where $\bone_n := ( 1, \ldots, 1 )^T$.
That is,
\[
\textrm{if } v \in \R^n \textrm{ is such that } v^T \bone_n = 0
\textrm{ then } v^T A v \geq 0.
\]
The authors showed in \cite[Theorem 2.9]{fitzgerald} that an entrywise
preserver $f[-]$ of this class of conditionally positive matrices
corresponds to a function that differs from being absolutely monotone by
a constant:
\begin{equation}\label{Econdpos}
f( x ) = \sum_{n = 0}^\infty c_n x^n \qquad \textrm{for all } x \in \R,
\textrm{ where } c_n \geq 0 \textrm{ for all } n \geq 1.
\end{equation}
From the perspective of the present article, conditionally positive
matrices are those real symmetric matrices with at most one negative
eigenvalue, with the negative eigenspace (if it exists) constrained to
equal~$\R \bone_n$.
If this constraint on the eigenspace is removed, Theorem~\ref{Tpreserver}
shows that the class of preservers shrinks dramatically.
The high level of rigidity for preservers of negative spectral multiplicity is akin to
that occurring for preservers of totally positive and totally
non-negative kernels: see the recent work \cite{BGKP-TN} for further
details on the latter.

Another way to view Schoenberg's Theorem~\cite{Schoenberg42} is as the
description of the entrywise transforms that preserve the class of
correlation matrices of vectors in Hilbert space. 
A completely parallel theory is developed in Section~\ref{Spontryagin},
providing the classification of self transforms of Gram matrices of
vectors belonging to a Pontryagin space (that is, a Hilbert space endowed
with an indefinite sesquilinear form with a finite number of negative
squares \cite{Azizov-Iokhvidov}).

Our next step involves relaxing the conditions appearing in
Theorem~\ref{Tpreserver}, by only imposing an upper bound on the number
of negative eigenvalues. In other words, we seek to classify the
endomorphisms of the closure
\begin{equation}
\overline{\cS_n^{(k)}}(I) := \bigcup_{j = 0}^k \cS_n^{(j)}(I), \qquad %
\textrm{where } n \geq 1.
\end{equation}
Note that the domain remains $I$ and not its closure~$\overline{I}$.
Similarly to before, we let $\overline{\cS_n^{(k)}}$ serve as an
abbreviation for $\overline{\cS_n^{(k)}}( \R )$.

Once again, if $k=0$ then this is just Schoenberg's
Theorem~\ref{Tschoenberg}, which yields a large class of transforms.
In contrast, if $k>0$ then we again obtain a far smaller class.

\begin{theorem}\label{Tatmostk}
Let $I := (-\rho,\rho)$, where $0 < \rho \leq \infty$, and let $k$ be a
positive integer. Given a function $f : I \to \R$, the following are
equivalent.
\begin{enumerate}
\item The entrywise transform $f[-]$ sends $\overline{\cS_n^{(k)}}(I)$ to
$\overline{\cS_n^{(k)}}$ for all $n \geq k$.
\item The function $f$ is either linear and of the form
$f( x ) \equiv f( 0 ) + c x$, where $f(0) \geq 0$ and $c > 0$, or
constant, so that $f(x) \equiv d$ for some $d \in \R$.
\end{enumerate}
\end{theorem}

Theorems~\ref{Tinertia},~\ref{Tpreserver} and~\ref{Tatmostk} are
negativity-preserving results. All three statements turn out to be
consequences of the following unifying theorem that we prove in
Section~\ref{S2} below.

\begin{utheorem}\label{Tstrongest}
Let $I := (-\rho,\rho)$, where $0 < \rho \leq \infty$, and let $k$ and
$l$ be positive integers. Given a function $f : I \to \R$, the
following are equivalent.
\begin{enumerate}
\item The entrywise transform $f[-]$ sends
$\overline{\cS_n^{(k)}}(I)$ to $\overline{\cS_n^{(l)}}$ for all $n \geq k$.
\item The entrywise transform $f[-]$ sends $\cS_n^{(k)}(I)$ to
$\overline{\cS_n^{(l)}}$ for all $n \geq k$.
\item Exactly one of the following occurs:
\begin{enumerate}
\item the function $f$ is constant, so that $f(x) \equiv d$
for some $d \in \R$;
\item it holds that $l \geq k$ and $f$ is linear, with
$f( x ) \equiv f( 0 ) + c x$, where $c>0$ and also $f(0) \geq 0$
if $l = k$.
\end{enumerate}
\end{enumerate}
If instead $k \geq 1$ and $l = 0$ then the entrywise transform
$f[-]$ sends $\overline{\cS_n^{(k)}}(I)$ (and so $\cS_n^{(k)}(I)$)
to $\overline{\cS_n^{(0)}} = \cS_n^{(0)}$
for all $n \geq k$ if and only if $f(x) \equiv c$ for some $c \geq 0$.

Finally, if $k = 0$ and $l \geq 1$ then the entrywise transform $f[-]$
sends $\cS_n^{(0)}(I)$ to $\overline{\cS_n^{(l)}}$ for all $n \geq 1$
if and only if
\[
f( x ) = \sum_{n = 0}^\infty c_n x^n \qquad \textrm{for all } x \in ( -\rho, \rho ), %
\textrm{ where } c_n \geq 0 \textrm{ for all } n \geq 1.
\]
\end{utheorem}

Note that setting $k = l = 0$ in Theorem~\ref{Tstrongest}(1)
(the missing case) gives exactly hypothesis~(1) of
Schoenberg's Theorem~\ref{Tschoenberg}. 

The class of functions identified in Theorem~\ref{Tstrongest} when $k = 0$
and $l \geq 1$  is independent of~$l$ and coincides with the
dimension-free entrywise preservers for two related but distinct
constraints:
(a) conditional positivity, as noted above (\ref{Econdpos}), and
(b) Loewner monotonicity, so that $f[ A ] - f[ B ] \in \cS^{(0)}$
whenever $A - B \in \cS^{(0)}$. 
The latter claim is a straightforward consequence of Schoenberg's
Theorem~\ref{Tschoenberg}; see \cite[Theorem~19.2]{Khare}.

\begin{remark}\label{Rnew}
A striking consequence of Theorem~\ref{Tstrongest} is as follows.
For any polynomial function $f$, if the negative spectral multiplicity of $f[ C ]$
is uniformly bounded above for any sufficiently large matrix $C$ then
$f$ cannot have any quadratic or higher-order terms.
The proof of this result is given in the next section, and uses
Sidon sets (also known as~$B_d$ sets) from additive combinatorics and number theory,
whose study was pioneered by Erd\"os and Tur\'an and developed by Chowla, among others.
\end{remark}

\subsection{Multi-variable transforms and non-balanced domains}

Given the results described above, it is natural to explore extensions in
two directions, aligned to previous work. In the sequel, for any integers
$m$ and $n$ with $m \leq n$, we let $[m : n ]$ denote the set
$\{ m, m + 1, \ldots, n \} = [ m, n ] \cap \Z$.
\begin{itemize}
\item \textit{Functions acting on matrices with entries in $I = ( 0, \rho )$
(positive entries) or in~$I = [ 0, \rho )$ (non-negative entries).}
For preservers of positive semidefiniteness, this problem was considered
by Loewner and Horn~\cite{horn} and Vasudeva~\cite{vasudeva79}
for the case $\rho = \infty$, and then in recent work \cite{BGKP-hankel}
for finite $\rho$. In each case, the class obtained consists of functions
represented by convergent power series with non-negative Maclaurin
coefficients.

\item \textit{Functions acting on $m$-tuples of matrices.}
A function $f: I^m \to \R$ acts entrywise on $m$-tuples of matrices with
entries in $I$: if $B^{(p)} = ( b^{(p)}_{i j} )$ is an $n \times n$
matrix for $p = 1$, $\ldots,$ $m$ then the  $n \times n$ matrix 
$f[ B^{(1)}, \ldots, B^{(m)} ]$ has $( i, j )$ entry
\[
f[ B^{(1)}, \ldots, B^{(m)} ]_{i j} = f( b^{(1)}_{i j}, \ldots, b^{(m)}_{i j} ) %
\qquad \textrm{for all } i, j \in [ 1 : n ].
\]
In this case, the classification of preservers in the
positive-semidefinite setting was achieved by FitzGerald, Micchelli and
Pinkus \cite{fitzgerald} when~$I = \R$, and then in our recent work
\cite{BGKP-hankel} over smaller domains.
\end{itemize}

Given a multi-index $\balpha = ( \alpha_1, \ldots, \alpha_m ) \in \Z_+^m$,
where $\Z_+ = \{ 0, 1, 2, \ldots \}$ is the set of non-negative integers,
and a point $\bx = ( x_1, \ldots, x_m ) \in \R^m$, we use the standard
notation $\bx^\balpha := x_1^{\alpha_1} \cdots x_m^{\alpha_m}$.

\begin{theorem}[\cite{BGKP-hankel}]\label{T000}
Let $I = (-\rho, \rho)$, $(0,\rho)$ or $[0,\rho)$, where
$0 < \rho \leq \infty$, and let $m$ be a positive integer. The function
$f : I^m \to \R$ acts entrywise to send $m$-tuples of positive
semidefinite matrices with entries in $I$ of arbitrary size to the set of
positive semidefinite matrices if and only if $f$ is represented on~$I^m$
by a convergent power series with non-negative coefficients:
\begin{equation}
f( \bx ) = \sum_{\balpha \in \Z_+^m} c_\balpha \bx^\balpha \qquad %
\textrm{for all } \bx \in I^m, \textrm{ where } %
c_\balpha \geq 0 \textrm{ for all } \balpha.
\end{equation}
\end{theorem}

Below, we extend this result to the complete classification of
negativity-preserving transforms acting on tuples of matrices, in the
spirit of the one-variable results above, over the three types of domain:
$I = (-\rho,\rho)$, $(0,\rho)$ or $[0,\rho)$, where
$0 < \rho \leq \infty$. The key to this is a multi-variable strengthening
of Theorem~\ref{Tstrongest} which also applies to these three different
types of domain.
The proof of this result, Theorem~\ref{Tmain}, appears in
Sections~\ref{Smulti} and~\ref{S1sided}, together with the necessary
supporting results and subsequent corollaries: Section~\ref{Smulti} is
concerned with the extension of Theorem~\ref{Tstrongest} to several
variables and Section~\ref{S1sided} then allows the restriction of $I$
from $( -\rho, \rho )$ to $( 0, \rho)$ or $[ 0, \rho )$.

\begin{notation}
In Theorem~\ref{Tstrongest}, the parameters $k$ and $l$ control the
degree of negativity in the domain and the co-domain, respectively. In
the multi-variable setting, the domain parameter~$k$ becomes an $m$-tuple
of non-negative integers $\bk = ( k_1, \ldots, k_m )$.
Given such a $\bk$, we may permute the entries so that any zero entries
appear first: more formally, there exists
$m_0 \in [ 0 : m ]$ with $k_p = 0$ for $p \in [ 1 : m_0 ]$
and $k_p \geq 1$ for $p \in [ m_0 + 1 : m ]$. We say that $\bk$ is
{\it admissible} in this case and
let $k_{\max} := \max\{ 1, k_p : p \in [ 1 : m ] \}$,
\[
\cS_n^{(\bk)}( I ) := \cS_n^{(k_1)}( I ) \times \cdots \times \cS_n^{(k_m)}( I ), %
\quad \text{and} \quad \overline{\cS_n^{(\bk)}}( I ) := %
\overline{\cS_n^{(k_1)}}( I ) \times \cdots \times \overline{\cS_n^{(k_m)}}( I ).
\]
\end{notation}

We now provide the result that unifies both Schoenberg's theorem and
Theorem~\ref{Tstrongest}. Its necessarily technical statement shows a
separation of variables, combining additively two kinds of preservers:
the rich class of multi-variable power series on one side and the rigid
family of homotheties on the other.

\begin{utheorem}[A Schoenberg-type theorem with negativity constraints]\label{Tmain}
Let $I := ( -\rho, \rho )$, $( 0, \rho )$, or $[ 0, \rho )$,
where $0 < \rho \leq \infty$, let $l$ and $m$ be non-negative integers,
with $m \geq 1$, and let $\bk \in \Z_+^m$ be an admissible tuple.
Given any function $f : I^m \to \R$, the following are equivalent.
\begin{enumerate}
\item The entrywise transform $f[ - ]$ sends $\overline{\cS_n^{(\bk)}}(I)$
to $\overline{\cS_n^{(l)}}$ for all $n \geq k_{\max}$.

\item The entrywise transform $f[ - ]$ sends $\cS_n^{(\bk)}(I)$ to $\overline{\cS_n^{(l)}}$
for all $n \geq k_{\max}$.

\item There exists a function $F : (-\rho,\rho)^{m_0} \to \R$
and a non-negative constant $c_p$ for each $p \in [ m_0 + 1 : m ]$
such that
\begin{enumerate}
\item we have the representation
\begin{equation}\label{Epresrep2}
f( \bx ) = F( x_1, \ldots, x_{m_0} ) + \sum_{p = m_0 + 1}^m c_p x_p
\qquad \textrm{for all } \bx \in I^m,
\end{equation}
\item the function
$\bx' := ( x_1, \ldots, x_{m_0} ) \mapsto F( \bx' ) - F( \bzero_{m_0} )$
is absolutely monotone, that is, it is represented on $I^{m_0}$
by a convergent power series with all Maclaurin coefficients
non-negative, and
\item we have the inequality
\[
\bone_{F( \bzero ) < 0} + \sum_{p : c_p > 0} k_p \leq l.
\]
\end{enumerate}
\end{enumerate}
\end{utheorem}

\noindent [Here and below, the quantity $\bone_P$ has the value $0$
if the proposition $P$ is false and $1$ if it is true.]

Theorem~\ref{Tmain} unifies all of the potentially disparate cases for
$\bk$ and $l$, such as those considered in the one-variable
Theorem~\ref{Tstrongest}, into one set of assertions; the multivariable
formulation seems to add clarity to the situation.
For instance, if $\bk = \bzero$ and $l > 0$ then (3)(c) is vacuously true,
whereas if $l = 0$ then condition (3)(c) requires that
$f( \bzero_m ) = F( \bzero_{m_0} ) \geq 0$.
Moreover, when $\bk = \bzero$ and $l = 0$
we recover Theorem~\ref{T000} and, in particular,
Schoenberg's theorem. Similarly, if $m = 1$ and $I = ( -\rho, \rho )$
then we recover Theorem~\ref{Tstrongest}.

Given previous results, the fact that the transforms classified by
Theorem~\ref{Tmain} are real analytic is to be expected, but
their exact structure is surprising.

Theorem~\ref{Tmain} is the building block we use to obtain the
classification of negativity-preserving transforms in several variables.
The classes of transforms do not depend on the choice of one-sided or
two-sided domains, akin to the one-variable setting. 

We complete the paper by proving complex-valued counterparts to the
preceding statements, now applying to Hermitian matrices. In this case,
the complex analogue of Schoenberg's theorem in one variable was proved
by Herz \cite{Herz63}; the multivariable result was obtained by
FitzGerald, Micchelli and Pinkus:

\begin{theorem}[\cite{fitzgerald}]\label{Tcomplex}
Let $m$ be a positive integer. The function $f : \C^m \to \C$ acts
entrywise to send $m$-tuples of positive semidefinite complex Hermitian
matrices to the set of positive semidefinite matrices if and only if $f$
is represented on $\C^m$ by a convergent power series in
$\bz = (z_1, \dots, z_m)$ and $\overline{\bz} = (\overline{z_1}, \dots, \overline{z_m})$ with non-negative coefficients:
\begin{equation}\label{Efmp}
f( \bz ) = \sum_{\balpha, \bbeta \in \Z_+^m} c_{\balpha, \bbeta}
\bz^\balpha \overline{\bz}^\bbeta \qquad %
\textrm{for all } \bz \in \C^m, \textrm{ where } %
c_{\balpha, \bbeta} \geq 0 \textrm{ for all } \balpha, \bbeta.
\end{equation}
\end{theorem}

With Theorem~\ref{Tcomplex} at hand, we provide the complex analogue of
our main result.

\begin{utheorem}\label{Tmain-complex}
Let $l$ and $m$ be non-negative integers, with $m \geq 1$, and let
$\bk \in \Z_+^m$ be an admissible tuple. Given any function
$f : \C^m \to \C$, the following are equivalent.
\begin{enumerate}
\item The entrywise transform $f[ - ]$ sends $\overline{\cS_n^{(\bk)}}(\C)$
to $\overline{\cS_n^{(l)}}(\C)$ for all $n \geq k_{\max}$.

\item The entrywise transform $f[ - ]$ sends $\cS_n^{(\bk)}(\C)$ to
$\overline{\cS_n^{(l)}}(\C)$ for all $n \geq k_{\max}$.

\item There exists  a function $F  :  \C^{m_0} \to \C$ and
non-negative constants $c_p$ and $d_p$ for each $p \in [ m_0 + 1 : m ]$
such that
\begin{enumerate}
\item we have the representation
\begin{equation}\label{Epresrep3}
f( \bz ) = F( z_1, \ldots, z_{m_0} ) +
\sum_{p = m_0 + 1}^m \bigl( c_p z_p + d_p \overline{z_p} \bigr)
\qquad \textrm{for all } \bz \in \C^m,
\end{equation}
\item the function $\bz' := (z_1, \dots, z_{m_0}) \mapsto F( \bz' ) - F( \bzero_{m_0} )$ is
represented on $\C^{m_0}$ by a convergent power series in $\bz'$ and $\overline{\bz'}$
with non-negative coefficients, as in~\tup{(\ref{Efmp})}, and
\item we have the inequality
\[
\bone_{F( \bzero ) < 0} + %
\sum_{p : c_p > 0} k_p + \sum_{p : d_p > 0} k_p \leq l
\]
and the constant $f( \bzero_m ) = F( \bzero_{m_0} )$ is real.
\end{enumerate}
\end{enumerate}
\end{utheorem}

Although it may seem elementary, a central technique for narrowing
down the class of negativity-preserving transforms to the hyper-rigid
forms listed above is to test them on non-orthogonal ``weighted sums of
squares", that is, linear combinations of carefully chosen rank-one
matrices. Starting with Horn's landmark dissertation \cite{horn}, this is
a key ingredient in almost all structured-matrix-preserver studies
\cite{Khare}. A second key technique in our proofs is the inflation and
deflation of symmetric matrices along isogenic blocks
\cite{BGKP-strata}.

\subsection{Organization of the paper}

In Section~\ref{S2}, we prove Theorem~\ref{Tstrongest} and then deduce
from it Theorems~\ref{Tinertia}, \ref{Tpreserver} and \ref{Tatmostk}.
Working in the one-variable setting and with the two-sided domain
$I = (-\rho,\rho)$ allows us to introduce several key ideas and techniques in
less complex circumstances, and these will then be employed in the
multi-variable setting and with one-sided domains.

We next explore the territory of Pontryagin space and classify in
Section~\ref{Spontryagin} the entrywise transforms of indefinite Gram
matrices in this environment.

In Sections~\ref{Smulti}, \ref{S1sided}, and \ref{Scomplex} we
prove the several-variables results mentioned above, first over the
two-sided product domain, then on the one-sided versions,
and finally, over~$\mathbb{C}$. In each section, this is followed by the
classification of negativity preservers, extending the results obtained
in the single-variable case.
We conclude with an appendix that proves a multi-variable Bernstein
theorem, asserting that absolutely monotone functions on $( 0, \rho )^m$
necessarily have power-series representations with non-negative Maclaurin
coefficients.

\subsection{Notation}

We let $\R$ denote the set of real numbers and
$\Z_+ = \{ 0, 1, 2, \ldots \}$ the set of non-negative integers.
Given $a$, $b \in \Z_+$ with $a < b$, we let
$[ a : b ] := [ a, b ] \cap \Z_+$. If~$\rho = \infty$ then $\rho / a$
and $\rho - a$ also equal $\infty$, for any finite $a > 0$.
We also set $0^0 := 1$.

\section{Inertia preservers for matrices with real entries}\label{S2}

We now begin to address the results appearing in the introduction.
In this section, we will consider functions with domain
$I = ( -\rho, \rho )$.
The proofs below involve several key ideas:
\begin{itemize}
\item[(a)] the translation $g$ of $f$, where $g( x ) := f( x ) - f( 0 )$;
\item[(b)] a ``replication trick'' that we will demonstrate shortly
(see (\ref{Ereptrick}) and thereafter);
\item[(c)] a result from previous work \cite{GKR-lowrank},
that functions sending positive matrices of rank at most $k$ to ones
of rank at most $l$ are necessarily polynomials;
\item[(d)] the use of Sidon sets (also known as $B_d$ sets) from number theory and
additive combinatorics;
\item[(e)] the following lemma.
\end{itemize}

\begin{lemma}\label{Lblock}
Let $h: I \to \R$ be absolutely monotone, where $I := ( -\rho, \rho )$ and
$0 < \rho \leq \infty$, and let
$C := \begin{bmatrix} A & B \\ B & A \end{bmatrix}$,
where $A$, $B \in \cS_n^{(0)}(I)$ are positive semidefinite matrices.
\begin{enumerate}
\item If $A - B \in \cS_n^{(k)}$ for some non-negative integer $k$ then
$C \in \cS_{2 n}^{(k)}(I)$.
\item If $h[ C ] \in \overline{\cS_{2n}^{(l)}}$ for some
non-negative integer $l$ then $h[A] - h[B] \in \overline{\cS_n^{(l)}}$.
\end{enumerate}
\end{lemma}
\begin{proof}
Note first that
\begin{equation}\label{Eblock}
J^T \begin{bmatrix} A & B \\ B & A \end{bmatrix} J = %
\begin{bmatrix} A + B & 0 \\ 0 & A - B \end{bmatrix}, \quad \textrm{where } %
J := \frac{1}{\sqrt{2}}\begin{bmatrix} \Id_n & -\Id_n \\ \Id_n & \Id_n \end{bmatrix}
\end{equation}
is orthogonal and $\Id_n$ is the $n \times n$ identity matrix.
As $A + B$ is positive semidefinite, it follows from (\ref{Eblock}) that
if $A - B \in \cS_n^{(k)}$ then $C \in \cS_{2n}^{(k)}(I)$.

Next, if $h[ C ] \in \overline{\cS_{2n}^{(l)}}$ then another
application of~(\ref{Eblock}) gives that
\[
\begin{bmatrix} h[A] + h[B] & 0 \\ 0 & h[A] - h[B] \end{bmatrix} \in
\overline{\cS_{2n}^{(l)}}.
\]
Schoenberg's Theorem~\ref{Tschoenberg} gives that $h[A] + h[B]$ is
positive semidefinite and therefore
$h[A] - h[B] \in \overline{\cS_n^{(l)}}$, as claimed.
\end{proof}

In addition to the key ideas, we will use a further two lemmas. The
first is the following basic consequence of Weyl's interlacing theorem.

\begin{lemma}\label{LWeyl}
Suppose $A \in \cS_n^{(k)}$ for some positive integer $n$ and
non-negative integer~$k$. If $B \in \cS_n^{(0)}$ has rank~1 then
$A + B \in \cS_n^{(k-1)} \cup \cS_n^{(k)}$
and $A - B \in \cS_n^{(k)} \cup \cS_n^{(k + 1)}$
$($where $\cS_n^{(-1)} := \emptyset$ and $\cS_n^{(n + 1)} := \emptyset)$.
\end{lemma}

\begin{proof}
Let $\lambda_1( C ) \leq \cdots \leq \lambda_n( C )$ denote the
eigenvalues of the $n \times n$ real symmetric matrix $C$, repeated
according to multiplicity, and let $\lambda_0( C ) := -\infty$ and
$\lambda_{n + 1}( C ) := \infty$. 

We note first that $C \in \cS_n^{(k)}$ if and only if
$0 \in ( \lambda_k( C ), \lambda_{k + 1}( C ) ]$,
for $k = 0$, $1$, $\ldots,$ $n$.

If $A$ and $B$ are as in the statement of the lemma, then the following
inequalities are a corollary \cite[Corollary~4.3.9]{HJ} of Weyl's
interlacing theorem:
\[
\lambda_1(A) \leq \lambda_1(A+B) \leq \lambda_2(A) \leq \cdots \leq
\lambda_n(A) \leq \lambda_n(A+B).
\]
It follows that
$0 \in ( \lambda_k( A ), \lambda_{k + 1}( A ) ] = I_1 \cup I_2$, where
$I_1 := ( \lambda_k( A ), \lambda_k( A + B ) ]$
and $I_2 := ( \lambda_k( A + B ), \lambda_{k + 1}( A ) ]$.

Then either $k \geq 1$ and
$0 \in I_1 \subseteq( \lambda_{k - 1}( A + B ), \lambda_k( A + B ) ]$,
so $A + B \in \cS_n^{(k - 1)}$,
or $0 \in I_2 \subseteq( \lambda_k( A + B ), \lambda_{k + 1}( A + B ) ]$,
so $A + B \in \cS_n^{(k)}$. The first claim follows.

For the second part, note that $A - B \in \cS_n^{(l)}$ for some
$l \in \{ 0, \ldots, n \}$, so the previous working gives that
$A = ( A - B ) + B \in \cS_n^{(l - 1)} \cup \cS_n^{(l)}$.
As $A \in \cS_n^{(k)}$, it follows that $l - 1 = k$ or $l = k$.
This completes the proof.
\end{proof}

The next lemma is a simple but useful observation about inertia.

\begin{lemma}\label{Lsylv}
Let $p$ and $q$ be non-negative integers and suppose
$\{ \bu_1, \ldots, \bu_p, \bv_1, \ldots, \bv_q \}$ is a set of linearly
independent vectors in $\R^n$, so that $n \geq p + q$. The matrix
\[
\bu_1 \bu_1^T + \cdots + \bu_p \bu_p^T - \bv_1 \bv_1^T - \cdots - \bv_q \bv_q^T
\]
has exactly $p$ positive and $q$ negative eigenvalues.
\end{lemma}

\begin{proof}
If necessary, we extend the given set to a basis by adding
vectors $\bw_1$, \ldots, $\bw_r$, where $r = n - p - q$. If
$C$ is the $n \times n$ matrix with these basis vectors as columns
and~$B$ is the $n \times n$ matrix defined in the statement of
the lemma then $B = C D C^T$, where
\[
D = \begin{bmatrix}
 \Id_p & 0 & 0 \\ 0 & -\Id_q & 0 \\ 0 & 0 & \bzero_{r \times r} 
\end{bmatrix}.
\]
The matrix $C$ has full rank, so is invertible, and it follows from
Sylvester's law of inertia \cite[Theorem~4.5.8]{HJ} that
$B$ and $D$ have the same inertia.
\end{proof}

\begin{proof}[Proof of Theorem~\ref{Tstrongest}]
We first show that (3) implies (2) (and this holds without any
restrictions on the domain $I \subset \R$).
 
If $f( x ) \equiv d$ for some $d \in \R$ then
$f[ A ] \in \overline{\cS_n^{(1)}} \subset \overline{\cS_n^{(l)}}$,
since the matrix $d \bone_{n \times n} = d \bone_n \bone_n^T$ has
nullity at least $n-1$ and the eigenvalue $n d$. This also gives one
implication for the final claim in the statement of the theorem, since
$f[ A ] = d \bone_{n \times n} \in \cS_n^{(0)}$ if $d \geq 0$. Moreover,
if $d < 0$ then $f[ A ] \in \cS_n^{(1)}$ and so $f$ does not map
$S_n^{(k)}$ into $\cS_n^{(0)}$ in this case.

Next, suppose $f( x ) = f( 0 ) + c x$ with $c > 0$, and let $A \in \cS_n^{(k)}$.
If $f( 0 ) \geq 0$ then
$f[ A ] = f( 0 ) \bone_{n \times n} + c A \in \cS_n^{(k-1)} \cup \cS_n^{(k)}$,
by Lemma~\ref{LWeyl}, and
so $f[ A ] \in \overline{\cS_n^{(l)}}$ as long as $l \geq k$.
If instead $f(0) < 0$ then, again by Lemma~\ref{LWeyl},
$f[A] = f(0) \bone_{n \times n} + cA$ has at most $k+1$ negative
eigenvalues and $k + 1 \leq l$ as long as $l > k$.

This shows $(3) \implies (2)$ and its $l = 0$ analogue in the
penultimate sentence of Theorem~\ref{Tstrongest}. For the $k = 0$
analogue in the final sentence, if $f$ has the prescribed form then
$g : x \mapsto f( x ) - f( 0)$ is absolutely monotone, so
$g[-] : \cS_n^{(0)}(I) \to \cS_n^{(0)}$ by
Schoenberg's Theorem~\ref{Tschoenberg}. Now Lemma~\ref{LWeyl} implies
that $f[-] : \cS_n^{(0)}(I) \to \overline{\cS_n^{(1)}} \subset
\overline{\cS_n^{(l)}}$, as required.

Next, we note that this working also yields the implication $(3)
\implies (1)$, including the $l=0$ analogue. Indeed, if (3) holds and
$k' \in [ 1 : k ]$ then the transform $f[ - ]$ maps
$\cS_n^{(k')}(I)$ to $\overline{\cS_n^{(l)}}$ by the $(3) \implies (2)$
result, whereas $f[ - ]$ maps $\cS_n^{(0)}(I)$ to
$\overline{\cS_n^{(l)}}$ by the $k = 0$ analogue.

That (1) implies (2) is immediate. It remains to show $(2) \implies (3)$.
Henceforth, we suppose that $I = (-\rho,\rho)$, where $0 < \rho
\leq \infty$, the non-negative integers $k$ and $l$ are not both zero,
and $f : I \to \R$ is such that $f[-]:\cS_n^{(k)}(I) \to
\overline{\cS_n^{(l)}}$ for all $n \geq k$, or for all $n \geq 1$ if $k =
0$.

To show that $f$ has the form claimed in each case, we proceed in a
series of steps.

\noindent \textit{Step 1:
Let $g: I \to \R$ be defined by setting $g( x ) := f( x ) - f( 0 )$.
Then $g$ is absolutely monotone.}

Indeed, by Lemma~\ref{LWeyl},
$g[-] : \cS_N^{(k)}(I) \to \overline{\cS_N^{(l+1)}}$ for all $N \geq k$.
If $A \in \cS_n^{(0)}(I)$, then the block-diagonal matrix
\[
D := ( -\eps \, \Id_k ) \oplus A^{\oplus ( l + 2 )} \in \cS_N^{(k)},
\]
where $\eps \in ( 0, \rho )$, there are $l+2$ copies of $A$ along the
block diagonal, and we have that $N = k + ( l + 2 ) n$.
It follows that
\begin{equation}\label{Ereptrick}
g[ D ] = ( g( -\eps ) \, \Id_k ) \oplus g[A]^{\oplus (l+2)} \in
\overline{\cS_N^{(l+1)}}.
\end{equation}
Thus, the block-diagonal matrix $g[A]^{\oplus (l+2)}$ can have at most
$l+1$ negative eigenvalues. This is only possible if $g[A]$ has no such
eigenvalues, which implies that $g[A]$ is positive semidefinite whenever
$A$ is. Thus, by Schoenberg's Theorem~\ref{Tschoenberg}, the function $g$
is absolutely monotone.
(This is the replication trick mentioned at the beginning of this section.)

If $k = 0$ we are now done, so henceforth we assume that $k \geq 1$.

\noindent \textit{Step 2:
If $g$ is as in Step~1 and $B \in \cS_n^{(0)}(I)$ has rank~$k$ then
$g[B]$ has rank at most $l$.}

Define an absolutely monotone function $h : I \to \R$ by setting
$h(x) := f(x) + |f(0)|$.  If $C \in \cS_n^{(k)}( I )$ then
$f[C] \in \overline{\cS_n^{(l)}}$, by assumption, and so
$h[C] = f[C] + |f( 0 )| \bone_{n \times n} \in \overline{\cS_n^{(l)}}$,
by Lemma~\ref{LWeyl}. Thus $h[-]:\cS_n^{(k)}(I) \to \overline{\cS_n^{(l)}}$
for all $n \geq k$. Applying Lemma~\ref{Lblock} with the matrices
$A = \bzero_{n \times n}$ and $B$ as above yields
\[
-g[ B ] = h[ \bzero_{n \times n} ] - h[ B ] \in \overline{\cS_n^{(l)}}.
\]
However,  the matrix $g[ B ]$ is positive semidefinite, by Schoenberg's
Theorem~\ref{Tschoenberg}, so it has no negative eigenvalues, and hence
has rank at most $l$.

We now resolve the $l = 0$ case in Step~3, before working with
$l>0$ in Steps~4 and~5.

\noindent \textit{Step 3:
If $f[-] : \cS_n^{(k)}(I) \to \cS_n^{(0)}$ for all
$n \geq k$ then $f$ is constant.}

Suppose the assumption holds; let $\eps \in ( 0, \rho )$ and note that
the function $g$ from Step~1 applied entrywise sends the matrix
$\eps \Id_k$ to $g( \eps ) \Id_k$.
By Step~2, the latter matrix has rank~$0$, so $g( \eps ) = 0$.
Since $g$ is absolutely monotone and vanishes on $( 0, \rho )$,
it must be zero, by the identity theorem. Hence $f$ is constant.

\noindent \textit{Step 4:
If $f$ is non-constant then $f$ is linear with positive slope.}

Let $g( x ) = \sum_{j = 1}^\infty c_j x^j$, where $c_j \geq 0$
for all $j$. We first recall \cite[Theorems~A and~B]{GKR-lowrank},
which imply that if $n \geq \max\{ k, l + 3 \}$ and the entrywise
transform $g[ - ]$ sends matrices in $\cS_n^{(0)}(I)$ with rank at most
$k$ to matrices with rank at most $l$, then $g$ is a polynomial. Hence it
follows from these results and Step~2 that~$g$ is a polynomial.
Moreover, by Lemma~\ref{LWeyl}, we have that
$g[-] : \cS_n^{(k)}(I) \to \overline{\cS_n^{(l+1)}}$ for all $n \geq k$.

We now suppose for contradiction that $c_r > 0$ for some $r \geq 2$.
Let $N \geq l+3$ be a positive integer and, for some $u_0 \in ( 0, 1 )$, let
\begin{equation}\label{EdNu}
\bu := ( 1, u_0, \ldots, u_0^{N - 1} )^T \in ( 0, 1 ]^N.
\end{equation}
For any tuple of distinct positive integers
$\bs = ( s_0, \ldots, s_{l + 2})$, the vectors
\[
\bu^{\circ s_i} = ( 1, u_0^{s_i}, \ldots, u_0^{(N - 1) s_i} )^T
\qquad \bigl( i \in [  0 : l + 2 ] \bigr)
\]
are linearly independent, as they form $l+3$ columns of a generalized
Vandermonde matrix; for the definition and properties of these, see,  for example, \cite[Section~4.1]{Khare}.
Our strategy is to apply $g[ - ]$ to a matrix of the
form $C := B \oplus -\eps \Id_{k - 1}$, where
\begin{equation}\label{EBs}
B := \eps \Bigl( -\epsilon \bu^{\circ s_0} ( \bu^{\circ s_0})^T + %
\sum_{i = 1}^{l + 2} \bu^{\circ s_i} ( \bu^{\circ s_i} )^T \Bigr) \in %
\cS_N^{(1)}
\end{equation}
as long as $\eps$ and $\epsilon$ are both positive, by Lemma~\ref{Lsylv}.
Then $C$ lies in $\cS_{N + k - 1}^{(k)}\bigl( ( 0, \rho ) \bigr)$
whenever $\eps \in \bigl( 0, \rho / ( l + 2 ) \bigr)$ and $\epsilon$
is positive and sufficiently small. It follows that
\[
g[ C ] = g[ B ] \oplus g( -\eps ) \Id_{k -1}
\]
has at most $l+1$ negative eigenvalues, whence so does the $N \times N$
matrix $g[ B ]$. We will show this to be false for a judicious choice of
$\bs$: its entries will form a generalized Sidon set, as explained in
Remark~\ref{Rsidon}.

We now let $s_i := ( d + 1 )^i$ for $i \in [ 0 : l + 2 ]$, where $d$ is
the degree of the polynomial~$g$, and suppose $N > ( d + 1 )^{l + 3}$.
We compute by multinomial expansion that
\begin{equation}\label{Ewm}
g[ B ] = \sum_{j = 1}^d c_j \eps^j \sum_{m_0 + \cdots + m_{l + 2} = j} %
( -\epsilon )^{m_0} \binom{j}{m_0, \ldots, m_{l + 2}} \bw_\bm \bw_\bm^T,
\end{equation}
where the sum is taken over
$\bm = ( m_0, \ldots, m_{l + 2} ) \in \Z_+^{l + 3}$, the vector
\[
\bw_\bm := \bu^{\circ \bs \cdot \bm} \in \R^N \quad \text{and} \quad %
\bs \cdot \bm = \sum_{i = 0}^{l + 2} m_i (d + 1)^i.
\]
By considering base-$(d + 1)$ representations of non-negative integers, it
is clear that the map
\[
[ 0 : d ]^{l + 3} \to \Z; \ %
\bm \mapsto \bs \cdot \bm = \sum_{i = 0}^{l + 2} m_i (d+1)^i
\]
is injective and its domain $[ 0 : d ]^{l + 3}$ has size $( d + 1)^{l + 3} < N$.
It follows by Vandermonde theory that the collection of vectors of the form
$\bw_\bm$ appearing in (\ref{Ewm}) is linearly independent
(again, see \cite[Section~4.1]{Khare}).
Hence, by Lemma~\ref{Lsylv}, the matrix $g[ B ]$ has precisely as many
negative eigenvalues as there are summands in (\ref{Ewm}) with a negative coefficient,
that is, those with $c_j > 0$ and $m_0$ odd. There are at least $l + 2$
such summands, when $j = d > 1$ and
\[
\bm \in %
\bigl\{ ( 1, d - 1, 0, \ldots, 0), \ (1, 0, d - 1, 0, \ldots, 0), \ \ldots, \ ( 1, \ldots, 0, d - 1) \bigr\}.
\]
Hence $g[ B ]$ has at least $l+2$ negative eigenvalues,
which is the desired contradiction.

\noindent \textit{Step 5:
Concluding the proof.}

We may assume $f( x ) = f( 0 ) + c x$, where $c > 0$, and we must show
that $l \geq k$ and that $f( 0 ) \geq 0$ if $l = k$.

If $l > k$ there is nothing to prove. If $l < k$ then we claim no such
function $f$ exists. Let $\{ v_1 := \bone_{k + 1}, v_2, \ldots, v_{k + 1} \}$
be an orthogonal basis of $\R^{k + 1}$, fix $\delta \in ( 0, \rho )$, and
choose a positive $\epsilon$ small enough so that the entries of the
matrix
\begin{equation}\label{Ectrex}
A := \delta \bone_{k + 1} \bone_{k + 1}^T - \epsilon \sum_{j=2}^{k+1} v_j v_j^T
\end{equation}
lie in $(0,\rho)$. Then $A \in \cS_{k + 1}^{(k)}(I)$ and
\[
f[ A ] = ( f( 0 ) + c \delta ) \bone_{k + 1} \bone_{k + 1}^T - %
c \epsilon \sum_{j=2}^{k+1} v_j v_j^T %
\in \cS_{k+1}^{(k)} \cup \cS_{k+1}^{(k+1)},
\]
as we know the eigenvalues explicitly. Hence $f[-]$ cannot send
$\cS_{k+1}^{(k)}(I)$ to $\overline{\cS_{k+1}^{(l)}}$ if $l<k$.

The final case to consider is when $l=k$ and $f( 0 ) < 0$, but the
counterexample~(\ref{Ectrex}) also works here if we insist that
$\delta < | f( 0 ) | / c$, whence $f( 0 ) + c \delta < 0$ and therefore
$f[ A ] \in \cS_k^{(k+1)}$; this shows that $f[A] \not\in \overline{\cS_{k+1}^{(l)}}$.
\end{proof}

\begin{remark}\label{Rsidon}
For completeness, we note that the specific integers $s_j = (d + 1)^j$
chosen in the proof of Step~4 above work because the map
$\bm \mapsto \bs \cdot \bm$ is injective. As long as~$N$ is taken
sufficiently large, we could instead have used any real tuple
$\bs$ that satisfies this condition; such positive-integer tuples,
for a fixed value of $\sum_j m_j$, are called $B_d$ or Sidon sets.
This is a very well studied notion in number theory and additive
combinatorics;
for early work in this area, see Singer~\cite{Singer}, Erd\"os and
Tur\'an~\cite{ET}, and Bose and Chowla~\cite{BC}.
\end{remark}

\begin{remark}
Atzmon and Pinkus studied entrywise transforms of rectangular matrices
that preserve bounds on rank: see \cite{Pinkus}.
\end{remark}

With Theorem~\ref{Tstrongest} at hand, we show the remaining results
above.

\begin{proof}[Proof of Theorems~\ref{Tinertia} and~\ref{Tpreserver}]
It is straightforward to verify that (2) implies (1) for both theorems.

Next, suppose $k \geq 1$. If $f[-]$ preserves the inertia of all matrices
in $\cS_n^{(k)}(I)$ then $f[-]$ sends $\cS_n^{(k)}(I)$ into
$\cS_n^{(k)}$.
If this holds for all $n$ then Theorem~\ref{Tstrongest} with $l=k$ gives
that either $f(x) \equiv d$ for some $d \in \R$, or $f$ is linear, so
$f(x) \equiv f(0) + c x$, with  $f(0) \geq 0$ and $c>0$.

If $k \geq 2$ then $f(x) \equiv d$ cannot send $\cS^{(k)}(I)$ to
$\cS^{(k)}$. Moreover, if $k=1$ then indeed $f(x) \equiv d$ preserves
$\cS^{(1)}$ for $d<0$ and does not do so if $d \geq 0$. Finally,
if $\eps \in ( 0 , \rho / 2 )$ then 
the matrix $\displaystyle \eps \begin{bmatrix} 1 & 2 \\ 2 & 1 \end{bmatrix}$
shows that no constant function preserves the inertia of all matrices in
$\cS^{(1)}(I)$. Thus, we now assume $f$ is of the form
$f( x ) = f( 0 ) + c x$, with $f( 0 ) \geq 0$ and $c > 0$, and show that $f(0) = 0$.

By hypothesis, if $\eps \in (0,\rho)$ then
\begin{equation}\label{Etemp}
f[-\eps \Id_k] = f( 0 ) \bone_{k \times k} - c \eps \Id_k \in \cS_k^{(k)}.
\end{equation}
As $f(0) \bone_{k \times k}$ has the eigenvalue $k f( 0 )$, so
$f[ - \eps \Id_k ]$ has the eigenvalue $k f( 0 ) -  c \eps$, which can be
made positive if $f( 0 ) > 0$ by taking $\eps$ sufficiently small. Thus $f(0) = 0$
and this concludes the proof for $k \geq 1$.

It remains to show that (1) implies (2) in Theorem~\ref{Tinertia} when
$k=0$. If $f[-]$ preserves the inertia of positive semidefinite matrices then, by
Schoenberg's Theorem~\ref{Tschoenberg}, the function
$f$ has a power-series representation $f(x) = \sum_{n=0}^\infty c_n x^n$
on $I$, with $c_n \geq 0$ for all $n \geq 0$.
Suppose $c_r > 0$ and $c_s > 0$ for distinct non-negative
integers $r$ and $s$. Applying $f[-]$ to the rank-one matrix
$A = \eps \bu \bu^T$, where $\bu = ( 1, 1 / 2 )^T$
and $\eps \in (0,\rho)$, and using the Loewner ordering, we see
that
\[
f[A] \geq c_r A^{\circ r} + c_s A^{\circ s} = %
c_r \eps^r \bu^{\circ r} ( \bu^{\circ r} )^T + %
c_s \eps^s \bu^{\circ s} ( \bu^{\circ s} )^T,
\]
where $\bu^{\circ r}$ and $\bu^{\circ s}$ are not
proportional. Hence $f[A]$ is positive definite, and therefore
non-singular, while $A$ is not. This contradicts the hypotheses, so 
the power series representing $f$ has at most one non-zero term.

Finally, we claim that $f$ is a homothety. If not, say $f(x) = c x^n$ for
$c > 0$ and $n \geq 2$, then we apply $f$ to the rank-$2$ matrix
$B = \eps ( \bone_{3 \times 3} + \bu \bu^T )$,
where $\bu = ( x, y , z)^T$ has distinct positive entries
and $\eps$ is positive and sufficiently small to ensure $B$ has
entries in $I$. The binomial theorem gives that
\[
B^{\circ n} = \eps^n
\sum_{j=0}^n \binom{n}{j} (x^j, y^j, z^j) (x^j, y^j, z^j)^T \geq
\eps^n \sum_{j=0}^2 \binom{n}{j} (x^j, y^j, z^j) (x^j, y^j, z^j)^T =: B',
\]
and the column space of $B'$ contains $(x^j, y^j, z^j)^T$ for
$j = 0$, $1$, and $2$. These three vectors are linearly independent, as a
Vandermonde determinant demonstrates, so $B'$ is positive definite, hence
non-singular. Then so is $B^{\circ n}$, while $B$ has rank $2$ by
construction. This contradicts the hypothesis, and so $n=1$ and
$f(x) = c x$ as claimed.
\end{proof}

\begin{proof}[Proof of Theorem~\ref{Tatmostk}]
If $f$ is constant then
$f[A] \in \overline{\cS_n^{(1)}} \subseteq \overline{\cS_n^{(k)}}$
for any $n \times n$ matrix $A$ and any positive integer $k$. Furthermore,
if $f(x) = f(0) + c x$, with $c>0$ and $f(0) \geq 0$,
then $f[A] = f( 0 ) \bone_{n \times n} + c A \in \overline{\cS_n^{(k)}}$ for
any $A \in \overline{\cS_n^{(k)}}$, by Lemma~\ref{LWeyl}.

Conversely, if $f[-]$ sends
$\overline{\cS_n^{(k)}}(I)$ to $\overline{\cS_n^{(k)}}$ for all
$n \geq 1$ then, in particular, it sends $\cS_n^{(k)}(I)$ to
$\overline{\cS_n^{(k)}}$.
Theorem~\ref{Tstrongest} now shows that $f$ has the form claimed.
\end{proof}

\section{Entrywise preservers of $k$-indefinite Gram
matrices}\label{Spontryagin}

\subsection{Gram matrices in Pontryagin space}

In analogy with the standard version of Schoenberg's theorem, we may
interpret the main result of this section as a classification of
entrywise preservers of finite correlation matrices in a Hilbert space
endowed with an indefinite metric. To be more specific,
we introduce the following terminology.

\begin{definition}
A \emph{Pontryagin space} is a pair $( H, J )$, where
$H$ is a separable real Hilbert space and $J : H \to H$ is 
a bounded linear operator such that $J = J^*$ and $J^2 = \Id_H$,
the identity operator on $H$. Note that $J$ is an isometric isomorphism.

We write $J = P_+ - P_-$, where $P_+$ and $P_-$ are orthogonal
projections onto the eigenspaces
$H_+ := \{ x \in H : J x =x \}$ and $H_- := \{ x \in H : J x = - x \}$,
respectively, so that $H = H_+ \oplus H_-$.

We say that the Pontryagin space has \emph{negative index $k$} if
$\dim H_- = k$. We assume henceforth that $k$ is positive and finite.
\end{definition}

Here we follow Pontryagin's original convention from \cite{Pontryagin},
that the negative index is taken to be finite, as opposed to that used by
Azizov and Iokhvidov \cite{Azizov-Iokhvidov}, where the positive index
$\dim H_+$ is required to be finite. In the same way that the sign choice
for Lorentz metric is merely a convention, there is no difference between
the theories which are obtained.

These spaces provide the framework for a still active, important branch
of spectral analysis. We refer the reader to the classic
monograph on the subject \cite{Azizov-Iokhvidov}.

\begin{definition}
The Pontryagin space $( H, J )$ carries a continuous symmetric
bilinear form $[ \cdot, \cdot ]$, where
\[
[ u, v ] := \langle u, J v\rangle \qquad \text{for all } u, v \in H.
\]
Given any $u \in H$, let $u_+ := P_+u $ and $u_- := P_- u$,
so that $u = u_+ + u_-$ and $J u = u_+ - u_-$. Then
\[
[u_+ + u_-, u_+ + u_- ] = \| u_+ \|^2 - \| u_- \|^2.
\]
\end{definition}

The analogy with Minkowski space endowed with the Lorentz metric is
obvious and it goes quite far \cite{Azizov-Iokhvidov}. We include, for
completeness, a proof of the following well known lemma.

\begin{lemma}\label{Lkappa}
\begin{enumerate}
\item Suppose $( H, J )$ is a Pontryagin space of negative index $k$. If
$( v_1, \ldots, v_n )$ is an $n$-tuple of vectors in the Hilbert space
$H$ and
\[
a_{i j} := [ v_i, v_j ] \qquad \textrm{for all } i, j \in [ 1 : n ]
\]
then the $n \times n$ \emph{correlation matrix}
$A = ( a_{i j } )_{i, j = 1}^n$ is real, symmetric,
and admits at most $k$ negative eigenvalues, counted with multiplicity.

\item Conversely, let $A = ( a_{i j} )_{i, j = 1}^n$ be a real symmetric
matrix with at most $k$ negative eigenvalues, counted with multiplicity.
There exists a Pontryagin space $( H, J )$ of negative index $k$ and an
$n$-tuple of vectors $( v_1, \ldots, v_n )$ in $H$ such that
\[
a_{i j} = [ v_i, v_j ] \qquad \textrm{for all } i, j  \in [ 1 : n ].
\]
\end{enumerate}
\end{lemma}

Having reminded the reader that we count eigenvalues with multiplicity,
we again leave this implicit.

\begin{proof}
(1) That $A$ is real and symmetric is immediate. Let $T: \R^n \to H$ be the
linear transform mapping the canonical basis vector~$e_j$
into $v_j$ for $j = 1$, $\ldots,$ $n$. Then
\[
a_{i j} = \langle e_i, T^\ast J T e_j \rangle %
\qquad \textrm{for all } i, j \in [ 1 : n ],
\]
and the self-adjoint operator $T^\ast J T = T^\ast P_{+}T - T^\ast P_{-} T$
has at most $k$ negative eigenvalues.

(2) Let $T$ denote the linear transformation of Euclidean space $\R^n$
corresponding to the given matrix $A$. Thus
\[
a_{i j} = \langle e_i, T e_j \rangle \qquad \textrm{for all } i, j \in [ 1 : n ],
\]
where $e_1$, $\ldots,$ $e_n$ are the canonical basis vectors in $\R^n$.
The spectral decomposition of~$T$ provides orthogonal projections
$Q_+$ and $Q_- = \Id - Q_+$ on $\R^n$ that commute with~$T$ and are such
that $T Q_{+} \geq 0$, $-T Q_- \geq 0$, and $r := \rank Q_- \leq k$.

Setting $H_+ := Q_+ \R^n$ and $H_- := ( Q_- \R^n ) \oplus \R^{k - r}$,
let $( H, J )$ be the Pontryagin space of negative index $k$
obtained by equipping $H = H_+ \oplus H_-$ with the map
$J : x_+ + x_- \mapsto x_+ - x_-$ for all $x_+ \in H_+$ and $x_- \in H_-$.

Let $v_i := \sqrt{T Q_{+}} e_i + \sqrt{-T Q_{-}} e_i$ for $i = 1$,
$\ldots,$ $n$. Since $\sqrt{T Q_+}$ maps $\R^n$ into $H_+$ and $\sqrt{-T
Q_-}$ maps $\R^n$ into $H_-$, we have that
\begin{align*}
[ v_i, v_j ] & = \langle \sqrt{T Q_{+}}e_i + \sqrt{-T Q_{-}}e_i, %
J ( \sqrt{T Q_{+}} e_j + \sqrt{-T Q_{-}}e_j ) \rangle \\
 & = \langle \sqrt{T Q_{+}} e_i,  \sqrt{T Q_{+}} e_j \rangle -  %
 \langle \sqrt{-T Q_{-}} e_i, \sqrt{-T Q_{-}} e_j \rangle\\
 & = \langle e_i, T Q_{+} e_j \rangle + \langle e_i, T Q_{-} e_j \rangle \\
 & =  \langle e_i, T (Q_{+} + Q_{-}) e_j \rangle \\
 & = a_{i j}. \qedhere
\end{align*}
\end{proof}

A small variation of the first observation above provides the following
stabilization result.

\begin{corollary}\label{Ckappa}
Let $( v_j )_{j = 1}^\infty$ be a sequence of vectors in a Pontryagin
space of negative index $k$. There exists a threshold $N$ such that the
number of negative eigenvalues of the Gram matrix
$\bigl( [ v_i, v_j ] \bigr)_{i, j=1}^n$ is constant for all $n \geq N$.
\end{corollary}

\begin{proof}
For any $n \geq 1$, the Gram matrix
$A_{[n]} :=  \bigl( [ v_i, v_j ] \bigr)_{i, j = 1}^n$
can have no more than~$k$ negative eigenvalues, by Lemma~\ref{Lkappa}(1).
Furthermore, by
the Cauchy interlacing theorem \cite[Theorem~4.3.17]{HJ}, the number of
negative eigenvalues in $A_{[n]}$ cannot decrease as $n$ increases.
The result follows.
\end{proof}

We now establish a partial converse of the above corollary,
employing an infinite-matrix version of Lemma~\ref{Lkappa}(2).

\begin{lemma}\label{LMihai}
Let $(a_{i j})_{i, j = 1}^\infty$ be an infinite real symmetric matrix
with the property that every finite leading principal submatrix of it has
at most $k$ negative eigenvalues.
Then there exists a sequence of vectors $( v_j )_{j =1}^\infty$ in a
Pontryagin space of negative index $k$ such that
\[
a_{i j} = [ v_i, v_j ] \qquad \textrm{for all } i, j \geq 1.
\]
\end{lemma} 

\begin{proof}
One can select successively positive weights $( w_j )_{j = 1}^\infty$
so that the infinite matrix
\[
B = ( b_{i j} )_{i, j = 1}^\infty := ( w_i w_j a_{i j} )_{i, j = 1}^\infty = %
\diag(w_1, w_2, \ldots) ( a_{i j} ) \diag( w_1, w_2, \ldots )
\]
has square-summable entries:
\[
\sum_{i, j =1}^\infty b_{i j}^2 < \infty.
\]
(To do this, choose $w_n$ so that the sum of the squares of the entries
of $B_{[n]} = ( b_{i j} )_{i, j = 1}^n$ that do not appear in
$B_{[n - 1]} = ( b_{i j} )_{i, j = 1}^{n - 1}$ is less than $2^{-n}$.)

Then the matrix $B$ represents a self-adjoint Hilbert--Schmidt operator
on~$\ell^2$ which we denote by $B$ as well. Since all finite leading
principal submatrices of this matrix have at most $k$ negative
eigenvalues, the operator $B$ has at most $k$ spectral points, counting
multiplicities, on $(-\infty, 0)$.
(Suppose otherwise, so that there exists a ($k + 1$)-dimensional subspace
$U$ of $\ell^2$ such that $\langle x, B x \rangle < 0$ whenever
$x \in U \setminus \{ 0 \}$. The truncation
$B_n := B_{[n]} \oplus \bzero_{\infty \times \infty}$ converges to $B$
in the Hilbert--Schmidt norm, so in operator norm, as $n \to \infty$,
and therefore $\langle x, B_n x \rangle < 0$ for all $x \in U \setminus \{ 0 \}$
and all sufficiently large $n$, a contradiction.)

Hence there exists an orthogonal projection $P_-$ on $\ell^2$ with
$r := \rank P_- \leq k$ that commutes with $B$ and is such that
$D := -B P_- \geq 0$ and $C := B P_+ \geq 0$, where $P_+ := I - P_-$. We
now proceed essentially as for the proof of Lemma~\ref{Lkappa}(2).

Define a Pontryagin space $( H, J )$ of negative index~$k$ by setting
$H := H_+ \oplus H_-$, where $H_+ := P_+ \ell^2$ and
$H_- := P_- \ell^2 \oplus \R^{k - r}$, and $J ( x_+ + x_- ) := x_+ - x_-$
whenever $x_+ \in H_+$ and $x_- \in H_-$.
If $( e_j )_{j = 1}^\infty$ is the canonical basis for $\ell^2 \subseteq H$
and $v_j := w_j^{-1} ( C^{1/2} + D^{1/2} ) e_j$ for all $j \geq 1$ then
\[
[ v_i, v_j ] = %
 w_i^{-1} w_j^{-1} \langle ( C^{1/2} + D^{1/2} ) e_i, ( C^{1/2} - D^{1/2} ) e_j \rangle =%
 w_i^{-1} w_j^{-1} \langle  e_i, ( C - D ) e_j \rangle = a_{i j}
\]
for all $i$, $j \geq 1$, as required.
\end{proof}

\subsection{Preservers of $k$-indefinite Gram matrices}

As a variation on Schoenberg's description of endomorphisms of
correlation matrices of systems of vectors lying in Hilbert space,
we now classify the entrywise preservers of the Gram matrices of
systems of vectors in a real Pontryagin space of negative index $k$.
We proceed by first introducing
some terminology and notation.

\begin{definition}
Given a non-negative integer $k$ and a sequence of vectors
$( v_j )_{j = 1}^\infty$ in a Pontryagin space, let
$a_{i j} := [ v_i, v_j ]$. The infinite real symmetric matrix
$A = ( a_{i j})_{i, j = 1}^\infty$ is a \emph{$k$-indefinite Gram matrix}
if the leading principal submatrix
$A_{[n]} := (a_{i, j})_{i, j = 1}^n$ has exactly $k$ negative
eigenvalues whenever $n$ is sufficiently large.

We denote the collection of all $k$-indefinite Gram matrices by $\sP_k$
and we let $\overline{\sP}_k := \bigcup_{j=0}^k \sP_j$,
the collection of $j$-indefinite Gram matrices for $j = 0$, $\ldots,$
$k$.
By Lemma~\ref{LMihai}, this is also the collection of all infinite
real symmetric matrices whose leading principal submatrices have
at most $k$ negative eigenvalues.
\end{definition}

Our next result classifies the entrywise preservers of $\sP_k$ and of
$\overline{\sP}_k$, in the spirit of Theorems~\ref{Tpreserver}
and~\ref{Tatmostk}. The entrywise preservers of Gram matrices in
Euclidean space are precisely the absolutely monotone functions, by
Schoenberg's Theorem~\ref{Tschoenberg}. However,
in the indefinite setting the set of preservers is much smaller.

\begin{theorem}\label{TPontryagin}
Fix a positive integer $k$ and a function $f : \R \to \R$.
\begin{enumerate}
\item The map $f[-] : \sP_k \to \sP_k$ if and only if $f$ is a
positive homothety, so that $f(x) \equiv c x$ for some $c>0$, 
or $k = 1$, in which case we may have a negative constant function,
so that $f(x) \equiv -c$ for some $c>0$.

\item The map $f[-] : \overline{\sP}_k \to \overline{\sP}_k$ if and only
if $f(x) \equiv d$ for some real constant $d$ or $f(x) = f(0) + cx$,
with $f( 0 ) \geq 0$ and $c > 0$.
\end{enumerate}
\end{theorem}

The key idea in the proof is to employ a construction
introduced and studied in detail in our recent work \cite{BGKP-strata},
which we now recall and study further.

\begin{definition}\label{Dinflation}
Let $\pi = \{ I_1, \ldots, I_m \}$ be a partition of the set of
positive integers $[ 1 : N ] = \{ 1, 2, \ldots, N \}$ into $m$ non-empty subsets,
so that $m \in [ 1 : N ]$. Define the \emph{inflation}
$\up : \R^{m \times m} \to \R^{N \times N}$ as the linear map such that
\[
\bone_{\{ i \} \times \{ j \}} \mapsto \bone_{I_i \times I_j},
\]
where $\bone_{A \times B}$ has $( p, q )$ entry $1$ if $p \in A$ and $q
\in B$, and $0$ otherwise.
\end{definition}

Thus, $\up$ sends every $m \times m$ matrix into one with blocks that are
constant on the rectangles defined by the partition $\pi$.

\begin{lemma}\label{Linertia}
Suppose $A$ is an $m \times m$ real symmetric matrix. Then $\up(A)$ has
the same number of positive eigenvalues and the same number of
negative eigenvalues, counted with multiplicity, as $A$ does.
\end{lemma}

\begin{proof}
Define the \emph{weight matrix} $\cW_\pi \in \R^{N \times m}$
to have $( i, j )$ entry $1$ if $i \in I_j$ and $0$ otherwise. As verified
in \cite{BGKP-strata}, we have that
\[
\up(A) = \cW_\pi A \cW_\pi^T\ \quad \textrm{for any } A \in \R^{m \times m}
\quad \textrm{and} \quad 
\cW_\pi^T \cW_\pi = \diag\bigl( |I_1|, \ldots, |I_m| \bigr).
\]
Since each set in $\pi$ is non-empty, the matrix $\cW_\pi$ has full
column rank and so may be extended to an invertible $N \times N$ matrix
$X_\pi = [ \cW_\pi \ \ C ]$ for some matrix
$C \in \R^{N \times ( N - m )}$. Then
\[
X_\pi \begin{bmatrix} A & 0 \\ 0 & 0 \end{bmatrix} X_\pi^T
= \cW_\pi A \cW_\pi^T = \up(A).
\]
Since $X_\pi$ is invertible, it follows from Sylvester's law of inertia
\cite[Theorem~4.5.8]{HJ}
that $\up(A)$ and $\begin{bmatrix} A & 0 \\ 0 & 0 \end{bmatrix}$ have the
same inertia, and the proof is complete.
\end{proof}

With Lemma~\ref{Linertia} at hand, we can prove
Theorem~\ref{TPontryagin}. First we introduce some notation.

\begin{definition}
Given an $n \times n$ real symmetric matrix
\begin{equation}\label{EAmat}
A = \begin{bmatrix} A_0 & \ba \\[1ex] \ba^T & a_{n n} \end{bmatrix}
\end{equation}
and a positive integer $m$, let $A_m := \up[m]( A )$ be the inflation of
$A$ according to the partition
\[
\pi_m := \bigl\{ \{ 1 \}, \ldots, \{ n - 1 \}, \{ n, \ldots, n - 1 + m \}
\bigr\},
\]
so that
\[
A_m = \begin{bmatrix}
 A_0 & \ba &  \cdots & \ba \\[1ex] \ba^T & a_{n n} & \cdots & a_{n n} \\
 \vdots & \vdots & \ddots & \vdots \\[1ex]
 \ba^T & a_{n n} & \cdots & a_{n n}
\end{bmatrix}.
\]
As $A \in \cS_n^{(k)}$ for some $k \in \{ 0, \ldots, n \}$, it follows
from Lemma~\ref{Linertia} that $A_m \in \cS_{n - 1 + m}^{(k)}$ for all
$m \geq 1$. Hence  there exists a $k$-indefinite Gram matrix $\wt{A}$
that can be considered as the direct limit of the sequence of matrices
$( A_m )_{m = 1}^\infty$, where
$\wt{A} = ( g_{i j} )_{i, j = 1}^\infty \in \sP_k$ is such that
$\wt{A}_{[N]} = ( g_{i j} )_{i, j = 1}^N = A_{N - n + 1}$ for all
$N \geq n$. In particular, the matrix $A$ is recoverable from $\wt{A}$,
since $\wt{A}_{[n]} = A$.

Below, the collection of all $k$-indefinite Gram matrices of the
form $\wt{A}$ for some $A$ as in (\ref{EAmat}) is denoted $\sP_k^\fin$.
\end{definition}

Note that
$f[ A_{[n]} ] = f[ A ]_{[n]}$ for all $n \geq 1$ and
$f[ A_m ] = f[ A ]_m$ for all $m \geq 1$, which implies that
$\wt{f[ A ]} = f[ \wt{A} ]$, whenever these quantities are well defined.

\begin{proof}[Proof of Theorem~\ref{TPontryagin}]

\begin{enumerate}
\item It is immediately seen that every positive homothety
preserves~$\sP_k$ and every negative constant sends $\sP_1$ to itself.
Furthermore, if~$f[-]$ preserves $\sP_k$ then it sends $\sP_k^\fin$
to~$\sP_k$. We now show that
\[
f[-] : \sP_k^\fin \to \sP_k \quad \textrm{if and only if} \quad
f[-] : \cS_n^{(k)} \to \cS_n^{(k)} \textrm{ for all } n \geq k.
\]
Given this, the result follows at once from Theorem~\ref{Tpreserver}.

Suppose $f[-]$ sends $\sP_k^\fin$ to $\sP_k$ and let $A \in \cS_n^{(k)}$
for some $n \geq k$. The Gram matrix $\wt{A} \in \sP_k^\fin$ so
$\wt{f[ A ]} = f[ \wt{A} ] \in \sP_k$ by assumption. Hence the
leading principal submatrix $\wt{f[ A ]}_{[N]} \in \cS_N^{(k)}$
for all sufficiently large $N$, so for some $N \geq n$,
but this matrix equals $f[ A ]_{N - n + 1}$, which has the same number of
negative eigenvalues as $f[ A ]$, by Lemma~\ref{Linertia}.
Thus $f[ - ]$ maps $\cS_n^{(k)}$ to itself.

Conversely, suppose $f[ - ]$ maps $\cS_n^{(k)}$ to itself for all $n \geq k$
and let $\wt{A} \in \sP_k^\fin$ for some $A \in \cS_n^{(k)}$.
We know that
$f[ A ] \in \cS_n^{(k)}$, by assumption, and therefore
$f[ \wt{A} ]_{[N]} = \wt{f[ A ]}_{[N]} = f[ A ]_{N - n +1} \in \cS_n^{(k)}$
for all $N \geq n$, again using Lemma~\ref{Linertia}.
Hence $f[ \wt{A} ] \in \sP_k^\fin \subseteq \sP_k$,
as claimed.

\item If $f( x ) \equiv d$ for some $d \in \R$ then
$f[ B ] = \wt{d} \in \overline{\sP}_1 \subseteq \overline{\sP}_k$
for all $k \geq 1$ and any infinite matrix $B$.
Similarly, if $f( x ) = f( 0 ) + c x$, with $f( 0 ) \geq 0$ and $c > 0$,
then $f[ B ] = \wt{f( 0 )} + c  B$,
so $f[ A ]_{[n]} = f( 0 ) \bone_{n \times n} + c A_{[n]} \in \overline{\cS_n^{(k)}}$
for all sufficiently large $n$ if $A \in \overline{\sP}_k$ , by Lemma~\ref{LWeyl}.
In both cases we see that $f[-] : \overline{\sP}_k \to \overline{\sP}_k$.
In turn, this condition implies that $f[-] : \sP_k^\fin \to \overline{\sP}_k$.
We now claim that if $f[-] : \sP_k^\fin \to \overline{\sP}_k$,
then $f$ is a real constant or linear of the above form.
Similarly to the previous part, this follows from
Theorem~\ref{Tstrongest}, given the following claim:
\[
f[-] : \sP_k^\fin \to \overline{\sP}_k \quad \text{if and only if} \quad
f[-] : \cS_n^{(k)} \to \overline{\cS_n^{(k)}}
\textrm{ for all } n \geq k.
\]
The proof of this equivalence follows the same lines as that in the
previous part, so the details are left to the interested reader.
\qedhere
\end{enumerate}
\end{proof}

\section{Multi-variable transforms with negativity constraints}\label{Smulti}

We now turn to the analysis of functions of several variables, acting on
tuples of matrices with prescribed negativity. In the present
section, we focus on functions of the form $f : I^m \to \R$,
where $I = (-\rho,\rho)$ for some $0 < \rho \leq \infty$.
The next section will address the cases where $I = ( 0, \rho )$ and
$I = [ 0, \rho )$. We will then turn to the complex setting in
Section~\ref{Scomplex}.

Recall that if $B^{(p)} = ( b^{(p)}_{i j} )$ is an $n \times n$ matrix
with entries in $I$ for all $p \in [ 1 : m ]$ then the function $f$ acts
entrywise to produce the
$n \times n$ matrix $f[ B^{(1)}, \ldots, B^{(m)} ]$ with $( i, j )$ entry
\[
f[ B^{(1)}, \ldots, B^{(m)} ]_{i j} = f(b^{(1)}_{i j}, \ldots, b^{(m)}_{i j}) %
\qquad \textrm{for all } i, j  \in [ 1 : n ].
\]

The negativity constraints on the domain are described by an $m$-tuple of
non-negative integers $\bk = ( k_1, \ldots, k_m )$. As noted in the
introduction, we may permute the entries of $\bk$  so that any zero
entries appear first: there exists
$m_0 \in [ 0 : m ]$ with $k_p = 0$ for $p \in [ 1 : m_0 ]$
and $k_p \geq 1$ for $p \in [ m_0 + 1 : m ]$.
In this case, the $m$-tuple $\bk$ is said to be admissible. We
let $k_{\max} := \max\{ k_p : p \in [ 1 : m ] \}$,
\[
\cS_n^{(\bk)}( I ) := \cS_n^{(k_1)}( I ) \times \cdots \times \cS_n^{(k_m)}( I ), %
\quad \text{and} \quad \overline{\cS_n^{(\bk)}}( I ) := %
\overline{\cS_n^{(k_1)}}( I ) \times \cdots \times \overline{\cS_n^{(k_m)}}( I ).
\]
The simplest generalization of the one-variable preserver problem would
involve taking $k_1 = \cdots = k_m = l$, but the more general problem
is more interesting (and also more challenging).
We now recall our generalization of Theorem~\ref{Tstrongest},
but break it into two parts for ease of exposition while proving it.

\begin{theorem}[The cases $\bk = \bzero$ and $l=0$]\label{Tto0}
Let $I = (-\rho,\rho)$, $(0,\rho)$ or $[0,\rho)$,
where $0 < \rho \leq \infty$, and let $\bk \in \Z_+^m$ be admissible.
Given a function $f : I^m \to \R$, the following are equivalent.
\begin{enumerate}
\item The entrywise transform $f[-]$ sends $\overline{\cS_n^{(\bk)}}(I)$ to
$\overline{\cS_n^{(0)}}$ for all $n \geq k_{\max}$.

\item The entrywise transform $f[-]$ sends $\cS_n^{(\bk)}(I)$ to $\cS_n^{(0)}$
for all $n \geq k_{\max}$.

\item The function $f$ is independent of $x_{m_0 + 1}$, $\ldots,$ $x_m$ and
is represented on $I^m$ by a convergent power series in the reduced tuple
$\bx' := ( x_1, \ldots, x_{m_0} )$, with all Maclaurin coefficients
non-negative.
\end{enumerate}
If, instead, $\bk = \bzero$ and $l \geq 1$ then $f[-]$ sends
$\cS_n^{(\bzero)}(I)$ to $\overline{\cS_n^{(l)}}$ if and only if the function
$f$ is  represented on~$I^m$ by a power series
$\sum_{\balpha \in \Z_+^m} c_\balpha \bx^\balpha$
with $c_\balpha \geq 0$ for all $\balpha \in \Z_+^m \setminus \{ \bzero \}$.
\end{theorem}

The case of entrywise positivity preservers corresponds to
taking $\bk = \bzero$ and $l = 0$. The analogue of Schoenberg's theorem holds,
as shown by FitzGerald, Micchelli, and Pinkus in \cite{fitzgerald}.
Before that, Vasudeva \cite{vasudeva79} had obtained the one-sided,
one-variable version of Schoenberg's theorem, which corresponds to taking
$m = 1$ and $I = ( 0, \infty )$. These results were extended to various smaller domains in
our previous work: see Theorem~\ref{T000} above, the proof of which employs
the multi-variable version of Bernstein's theorem on absolutely
monotone functions. This theorem was obtained by Bernstein \cite{Bernstein} for
$m = 1$, and then by Schoenberg~\cite{Schoenberg33} for $m = 2$. For $m > 2$,
see Appendix~\ref{Appendix}.

Schoenberg's theorem and these generalizations yield a large collection
of preservers. Theorem~\ref{Tto0} shows that a similarly rich class of
preservers is obtained even if one relaxes some (but not all) of the
negativity constraints, from none to some, provided that the codomain is
still required to be $\cS_n^{(0)}$.

We now complete the classification of the multi-variable
transforms for the remaining case of Theorem~\ref{Tmain}.
The class of functions obtained is given by a combination of the rich class
of absolutely monotone functions in the variables for which
the corresponding entries of~$\bk$ are zero
and the rigid class of positive homotheties in a subset of the remaining
variables.

\begin{theorem}[The case of $\bk \neq \bzero$ and $l > 0$]\label{Tmulti2}
Let $I := ( -\rho, \rho )$, $( 0, \rho )$ or $[ 0, \rho )$,
where $0 < \rho \leq \infty$,
let $\bk$ be admissible and not equal to~$\bzero$, and suppose
$l$ is a positive integer.
Given a function $f : I^m \to \R$, the following are equivalent.
\begin{enumerate}
\item The entrywise transform $f[ - ]$ sends
$\overline{\cS_n^{(\bk)}}(I)$ to $\overline{\cS_n^{(l)}}$
for all $n \geq k_{\max}$.

\item The entrywise transform $f[ - ]$ sends $\cS_n^{(\bk)}(I)$ to
$\overline{\cS_n^{(l)}}$ for all $n \geq k_{\max}$.

\item There exists  a function $F : (-\rho,\rho)^{m_0} \to \R$ and a
non-negative constant $c_p$ for each $p \in [ m_0 + 1 : m ]$ such that
\begin{enumerate}
\item we have the representation
\begin{equation}\label{Epresrep}
f( \bx ) = F( x_1, \ldots, x_{m_0} ) + \sum_{p = m_0 + 1}^m c_p x_p
\qquad \textrm{for all } \bx \in I^m,
\end{equation}
\item the function
$\bx' := ( x_1, \ldots, x_{m_0} ) \mapsto F( \bx' ) - F(
\bzero_{m_0} )$ is absolutely monotone, that is, it is represented
on $I^{m_0}$ by a convergent power series with all Maclaurin coefficients
non-negative, and
\item we have the inequality
\[
\bone_{F( \bzero ) < 0} + \sum_{p : c_p > 0} k_p \leq l.
\]%
\end{enumerate}
\end{enumerate}
\end{theorem}

In the above, we adopt the convention that if $m_0 = 0$ then
$F( x_1, \ldots, x_{m_0} )$ is the constant~$F( \bzero )$
and (3)(b) is vacuously satisfied.

\begin{remark}\label{Rrescale}
As asserted in \cite[Remark~9.16]{BGKP-hankel} for the $\bk = \bzero$
case, there is no greater generality in considering, for
Theorems~\ref{Tto0} and~\ref{Tmulti2}, domains of the form
\[
( -\rho_1, \rho_1 ) \times \cdots \times ( -\rho_m, \rho_m ), \qquad
\textrm{where } 0 < \rho_1, \ldots, \rho_m \leq \infty.
\]
For finite $\rho_p$, one can introduce the scaling
$x_p \mapsto x_p / \rho_p$, whereas for infinite $\rho_p$, one can
truncate to $( -N, N )$, scale and then use the identity theorem to
facilitate the extension to $( -\infty, \infty )$.
\end{remark}

The simplest multi-variable generalization mentioned above,
$k_1 = \cdots = k_m = l$, is a straightforward consequence of
Theorem~\ref{Tmulti2}. The classification obtained is rigid and very
close to what is seen in Theorems~\ref{Tpreserver} and~\ref{Tatmostk}.

\begin{corollary}\label{Cpreserver}
Let $I := ( -\rho, \rho )$, $( 0, \rho )$ or $[ 0, \rho )$,
where $0 < \rho \leq \infty$, let $k$ and $m$ be positive integers,
and let $f : I^m \to \R$.
\begin{enumerate}
\item The entrywise transform $f[-]$ sends $\cS_n^{(k \bone_m^T)}(I)$
to $\cS_n^{(k)}$ for all $n \geq k$ if and only if $f( \bx ) = c x_{p_0}$
for a constant $c > 0$ and some $p_0 \in [ 1 : m ]$, or,
when $k = 1$, we may also have $f( \bx ) \equiv -c$ for some $c > 0$.

\item The entrywise transform $f[-]$ sends
$\overline{\cS_n^{(k \bone_m^T)}}(I)$ to
$\overline{\cS_n^{(k)}}$ for all $n \geq k$
if and only if $f( \bx ) = c x_{p_0} + d$ for some $p_0 \in [ 1: m ]$,
with either $c = 0$ and $d \in \R$, or $c > 0$ and $d \geq 0$.
\end{enumerate}
\end{corollary}

It is not clear how to define inertia preservers when there are multiple
matrices to which an entrywise transform is applied. Hence we provide no
version of Theorem~\ref{Tinertia} in Corollary~\ref{Cpreserver}. However,
when we work over $(0,\rho)$ or $[0,\rho)$ in the next section, we will
give a proof of the analogue of Theorem~\ref{Tinertia} for these domains
in the $m=1$ case.

\begin{proof}
We first prove~(1). The reverse implication is readily verified; for the
forward assertion, apply Theorem~\ref{Tmulti2},
noting that $m_0 = 0$, to obtain that either
$f( \bx ) \equiv F( \bzero )$ or $f( \bx ) = F( \bzero ) + c x_{p_0}$ with
$F( \bzero ) \geq 0$ and $c > 0$.
Now $f \equiv F( \bzero )$ cannot take $\cS_n^{(k \bone_m^T)}(I)$ to
$\cS_n^{(k)}$ for any real $F( \bzero )$ if $k \geq 2$, and for
non-negative $F( \bzero )$ if $k = 1$. We next assume that
$f( \bx ) = F( \bzero ) + c x_{p_0}$; to complete the proof, we need to show
that $F( \bzero ) = 0$. Taking $B^{(1)} = \cdots = B^{(m)} = -\eps \Id_k$ for
$\eps \in ( 0, \rho )$, the working around~(\ref{Etemp}) gives
that~$F( \bzero ) = 0$.

Coming to~(2), the reverse implication is trivial for $c = 0$, and follows
from Lemma~\ref{LWeyl} if $c > 0$ and $d \geq 0$. The forward
implication follows immediately from Theorem~\ref{Tmulti2}.
\end{proof}

\subsection{The proofs}

The remainder of this section is devoted to establishing Theorems~\ref{Tto0}
and~\ref{Tmulti2} for the case $I = (-\rho,\rho)$.
In addition to the replication trick and multi-variable analogues of
various lemmas used above, a key ingredient is Lemma~\ref{Linertia}.
This result is indispensable for producing tuples of test matrices of the
same large dimension while preserving the number of negative eigenvalues.

The next proposition provides suitable
multi-variable versions of Steps~1 and~2 from the proof of Theorem~\ref{Tstrongest}.
In order to state it clearly, we introduce some notation.

\begin{notation}
Given $n \times n$ test matrices  $A^{(1)}$, $\ldots,$ $A^{(m)}$
and constants $\epsilon_1$, $\ldots,$ $\epsilon_m$, we let the $m$-tuples
$\bA := ( A^{(1)}, \ldots, A^{(m)} )$
and
\begin{equation}
\bA + \bepsilon \bone := %
( A^{(1)} + \epsilon_1 \bone_{n \times n}, \ldots, A^{(m)} + \epsilon_m \bone_{n \times n} ),
\end{equation}
and similarly for $\bB$ and $\bB + \bepsilon \bone$.

If $P = \{ p_1 < \cdots < p_k \}$ is a non-empty subset of $[ 1 : m ]$
then $\bA_P$ and $\bA_P + \bepsilon_P \bone$ denote the corresponding
$k$-tuples with entries whose indices appear in $P$, so that
$\bA_P = ( A^{(p_1)}, \ldots, A^{(p_k)} )$ and similarly for
$\bA_P + \bepsilon_P \bone$.

If $\balpha \in \Z_+^m$ then $\bA^{\circ \balpha}$ denotes the matrix
$( A^{(1)} )^{\circ \alpha_1} \circ \cdots \circ ( A^{(m)} )^{\circ \alpha_m}$,
where $\circ$ denotes the Schur product, so that $\bA^{\circ \balpha}$
has $( i, j )$ entry
\[
\bA^{\circ \balpha}_{i j} = ( A^{(1)}_{i j} )^{\alpha_1} \cdots ( A^{(m)}_{i j} )^{\alpha_m} %
\qquad \textrm{for all } i, j \in [ 1 : n ].
\]
A function $f$ with power-series representation
$f( \bx ) = \sum_{\balpha} c_\balpha \bx^\balpha$
acts on the $m$-tuple $\bA$ to give the $n \times n$ matrix
$f[ \bA ] = \sum_\balpha c_\balpha A^{\circ \balpha}$.
\end{notation}

The following result and its corollary are not required when considering
functions of a single variable but they are crucial to our arguments in
the multivariable setting.

\begin{proposition}\label{Pabsmon}
Let $I := (-\rho,\rho)$, $(0,\rho)$ or $[0,\rho)$, where
$0 < \rho \leq \infty$.
Suppose  $f : I^m \to \R$ is such that the entrywise transform $f[-]$
sends $\cS_n^{(\bk)}( I )$ to $\overline{\cS_n^{(l)}}$ for all
$n \geq k_{\max}$, where $\bk \in \Z_+^m$
and $l$ is a non-negative integer.
\begin{enumerate}
\item The function $f$ is represented on $I^m$ by a power series
$\sum_{\balpha \in \Z_+^m} c_\balpha \bx^\balpha$, where the coefficient
$c_\balpha \geq 0$ for all $\balpha \in \Z_+^m \setminus \{ \bzero \}$.

\item Suppose the matrices $A^{(1)}$, $\ldots,$ $A^{(m)}$,
$B^{(1)}$, $\ldots,$ $B^{(m)} \in \cS_n^{(0)}$
are such that $A^{(p)} - B^{(p)} \in \cS_n^{(k_p)}$ for all $p \in [ 1 : m ]$.
Fix $\epsilon_1$, $\ldots,$ $\epsilon_m \geq 0$ such that the entries of
$A^{(p)} + \epsilon_p \bone_{n \times n}$ and $B^{(p)} + \epsilon_p \bone_{n \times n}$
all lie in $I$ for every $p \in [ 1 : m ]$. Then
\[
f[ \bB + \bepsilon \bone ] - f[ \bA + \bepsilon \bone ]
\]
is a positive semidefinite matrix with rank at most $l$.
\end{enumerate}
\end{proposition}

\begin{proof}[Proof of Proposition~\ref{Pabsmon} for $I = ( -\rho, \rho )$]\hfill
\begin{enumerate}
\item Define $g$ on $I^m$ by setting $g( \bx ) = f( \bx ) - f( \bzero )$.
Given a constant $\eps \in ( 0, \rho )$ and matrices
$A_1$, $\ldots,$ $A_m \in \cS_n^{(0)}(I)$,
let $N := k_{\max} + ( l + 2 ) n$ and
\begin{equation}\label{Emultitest1}
B^{(p)} := -\eps \Id_{k_p} \oplus %
\bzero_{ ( k_{\max} - k_p ) \times ( k_{\max} - k_p )} \oplus %
A_p^{\oplus ( l + 2 )} \in \cS_N^{(k_p)}( I )
\end{equation}
for all $p \in [ 1 : m ]$. By hypothesis and Lemma~\ref{LWeyl},
the entrywise transform $g[-]$ sends $\cS_N^{(\bk)}( I )$ to
$\overline{\cS_N^{( l + 1 )}}$, so $g[ \bB ] \in \overline{\cS_N^{(l + 1)}}$.
As the block-diagonal matrix~$g[ \bB ]$ contains $l + 2$ copies of
the matrix $g[ \bA ]$ on the diagonal, this last matrix must be
positive semidefinite. It now follows from
Theorem~\ref{T000} that $g$ is absolutely monotone.

\item It follows from the previous part that the function
$h : \bx \mapsto f( \bx ) + |f( \bzero )|$ is absolutely monotone,
so $h[ \bA + \bepsilon \bone ]$ and $h[ \bB + \bepsilon \bone ]$
are positive semidefinite.

By Lemma~\ref{LWeyl}, $h[-]$ sends $\cS_n^{(\bk)}(I)$ to
$\overline{\cS_n^{(l)}}$ for all $n \geq k_{\max}$. Using~(\ref{Eblock}),
we see that
\[
C^{(p)} := \begin{bmatrix}
 A^{(p)} + \epsilon_p \bone_{n \times n} & B^{(p)} + \epsilon_p \bone_{n \times n} \\
 B^{(p)} + \epsilon_p \bone_{n \times n} & A^{(p)} + \epsilon_p \bone_{n \times n}
 \end{bmatrix} \in \cS_{2 n}^{(k_p)}( I )
\]
for all $p \in [ 1 : m ]$, so $h[ \bC ] \in \overline{\cS_{2 n}^{(l)}}$.
Hence, again using~(\ref{Eblock}), we have that
\[
f[ \bA + \bepsilon \bone ] - f[ \bB + \bepsilon \bone ] = %
h[ \bA + \bepsilon \bone ] - h[ \bB + \bepsilon \bone ] \in \overline{\cS_n^{(l)}}.
\]
By the Schur product theorem, any absolutely monotone function preserves Loewner monotonicity
when acting entrywise; see Theorem~\ref{Tschoenberg}.
Hence the matrix $f[ \bB + \bepsilon \bone ] - f[ \bA + \bepsilon \bone ]$
is positive semidefinite, and so has rank at most~$l$, by the previous working.
\qedhere
\end{enumerate}
\end{proof}

We now present an application of the previous proposition
after first introducing some notation. Given a set
$P = \{ p_1 < \cdots < p_k \} \subseteq [ 1 : m ]$
and $\bx = ( x_1, \ldots, x_m ) \in \R^m$, let $P' = [ 1 : m ] \setminus P$
and $\bx_P = ( x_{p_1}, \ldots, x_{p_k} )$. Further, let $| P |$
denote the cardinality of the set $P$, so that $| P | \in [ 0 : m ]$.

\begin{corollary}\label{Csample}
Let $I := ( -\rho, \rho )$, $( 0, \rho )$ or $[ 0, \rho )$, where
$0 < \rho \leq \infty$.
Suppose the function \mbox{$f : I^m \to \R$} is such that the entrywise transform
$f[-]$ sends $\cS_n^{(\bk)}( I )$ to $\overline{\cS_n^{(l)}}$ for all
$n \geq k_{\max}$, where $\bk \in \Z_+^m$ and $l$ is a non-negative
integer.
\begin{enumerate}
\item Given any $P \subseteq [ 1 : m ]$, we may write $f$ in the \emph{split representation}
\begin{equation}\label{Esplitf}
f( \bx ) = \sum_{\balpha_P \in \Z_+^{| P |}} c_{\balpha_P}( \bx_{P'} ) x_P^{\balpha_P} %
\qquad \text{for all } \bx \in I^m,
\end{equation}
where each $c_{\balpha_P}$ is a function on $I^{P'}$ such that
$\bx_{P'} \mapsto c_{\bzero_P}( \bx _{P'} ) - c_{\bzero_P}( \bzero_{P'} )$
and~$c_{\balpha_P}$ (with $\balpha_P \neq \bzero_P$)
are absolutely monotone.

\item For each $p \in [ 1 : m ]$, let $B^{(p)} \in \cS_n^{(0)}$ have
rank $k_p$ and suppose the non-negative constant $\epsilon_p$ is such
that
$B^{(p)} + \epsilon_p \bone_{n \times n}$ has all of its entries in $I$.
If $g : I^P \to \R$ is such that
\[
g( \bx_P ) = %
\sum_{\balpha_P \in \Z_+^{| P |} \setminus \{ \bzero \}} c_{\balpha_P}( \bepsilon_{P'} ) %
\bx_P^{\balpha_P} \qquad \textrm{for all } \bx_P \in I^P,
\]
where each $c_{\balpha_P}$ is as in (\ref{Esplitf}),
then the matrix $g[ \bB_P ]$ is positive semidefinite and has rank at most $l$.
\end{enumerate}
\end{corollary}

\begin{proof}[Proof of Corollary~\ref{Csample} for $I = (-\rho,\rho)$]
\begin{enumerate}
\item This is an immediate consequence of Proposition~\ref{Pabsmon}(1)
and the fact that $\bx^\balpha = \bx_{P'}^{\balpha_P'} \bx_P^{\balpha_P}$.

\item That $g[ \bB_P ]$ is positive semidefinite is a straightforward consequence
of the Schur product theorem.
Next, applying Proposition~\ref{Pabsmon}(2) with $ (A^{(p)}, B^{(p)}, \epsilon_p )$
there equal to
$( \epsilon_p' \bone_{n \times n}, B^{(p)} + \epsilon_p' \bone_{n \times n}, \epsilon_p')$
for all $p$, where $\epsilon_p' = \epsilon_p / 2$, it follows that
\[
C := f[ \bB + \bepsilon \bone ] - f[ \bepsilon \bone ]
\]
is positive semidefinite and has rank at most $l$. Thus, for the final claim
it suffices to show that the inequality $C \geq g[ \bB_P ]$ holds. Note first that
\[
c_{\alpha_P}[ \bB_{P'} + \bepsilon_{P'} \bone ] \geq %
c_{\alpha_P}[ \bepsilon_{P'} \bone ] \quad \textrm{for all } \alpha_P \in \Z_+^{| P |},
\]
by Loewner monotonicity, as in the proof of Proposition~\ref{Pabsmon}(2).
Again using the Schur product theorem, it follows that
\begin{align*}
C & = \sum_{\balpha_P \in \Z_+^{| P |}} \Bigl( %
c_{\balpha_P}[ \bB_{P'} + \bepsilon_{P'} \bone ] \circ %
 ( \bB_P + \bepsilon_P \bone )^{\circ \balpha_P} - %
 c_{\balpha_P}[ \bepsilon_{P'} \bone ] \circ %
 ( \bepsilon_P \bone )^{\circ \balpha_P} \Bigr) \\
 & \geq \sum_{\balpha_P \in \Z_+^{| P |}} c_{\balpha_P}[ \bepsilon_{P'} \bone ] \circ %
\Bigl( ( \bB_P + \bepsilon_P \bone )^{\circ \balpha_P} - %
( \bepsilon_P \bone )^{\circ \balpha_P} \Bigr)\\
 & = \sum_{\alpha_P \neq \bzero} c_{\balpha_P}( \bepsilon_{P'} ) \Bigl( %
( \bB_P + \bepsilon_P \bone )^{\circ \balpha_P} - %
( \bepsilon_P \bone )^{\circ \balpha_P} \Bigr) \\
 & \geq \sum_{\alpha_P \neq \bzero} c_{\balpha_P}( \bepsilon_{P'} ) %
 \bB_P^{\circ \balpha_P} \\
  & = g[ \bB_P ].
\qedhere
\end{align*}
\end{enumerate}
\end{proof}

With Proposition~\ref{Pabsmon} and Corollary~\ref{Csample} at hand, we
show the two theorems above.

\begin{proof}[Proof of Theorem~\ref{Tto0} for $I = (-\rho,\rho)$]
We begin by showing a cycle of implications for the three equivalent
hypotheses. The Schur product theorem gives that (3) implies (1) and
it is immediate that (1) implies (2). We now assume~(2) and deduce that
(3) holds.

Note first that the case $m_0 = m$ is Theorem~\ref{T000}, so we assume
henceforth that $m_0 < m$. Proposition~\ref{Pabsmon}(1) gives that
$\bx \mapsto f( \bx ) - f( \bzero )$ is absolutely monotone and we
claim that $f( \bzero ) \geq 0$, so that $f$ is itself absolutely monotone.
To see this, let $N = 1 + k_{\max}$ and $\eps \in ( 0, \rho )$, and set
\begin{equation}\label{Etestmatrix3}
B^{(p)} := \bzero_{( N - k_p )  \times ( N - k_p) } \oplus %
-\eps \Id_{k_p} \in \cS_N^{(k_p)}(I) \qquad \text{for all } p \in [ 1 : m ].
\end{equation}
By hypothesis, $f[ \bB ] \in \cS_N^{(0)}$ is positive
semidefinite. In particular, its $(1,1)$ entry $f( \bzero )$ is
non-negative.

To complete the first part of the proof, we now show that the power series
that represents $f$ contains no monomials involving any of $x_{m_0 + 1}$,
$\ldots,$ $x_m$.
We suppose without loss of generality that there is a monomial
containing $x_m$.
Let $a$, $b$, $2 \eps \in (0,\rho)$, with $a < b$, let
$N \geq 2 + k_{\max}$, and set
\[
B^{(m)}_0 := \begin{bmatrix} a & b \\ b & a \end{bmatrix} \oplus %
( 2 \eps \bone_{k_m \times k_m} - \eps \Id_{k_m} ) \oplus %
 \eps \Id_{N - k_m - 2} \in \cS_N^{(k_m)},
\]
since the first matrix has one positive and one negative eigenvalue and
$2 \eps \bone_{k_m \times k_m}$ has eigenvalue $0$ with
multiplicity $k_m - 1$ and eigenvalue $2 k_m \eps$ with
multiplicity $1$.

Similarly, for $p \in [ 1 : m - 1 ]$, let
\[
B'_p := ( 2 \eps \bone_{(k_p + 1) \times (k_p + 1)} - \eps \Id_{k_p + 1} ) \oplus %
\eps \Id_{N-k_p-2} \in \cS_{N - 1}^{(k_p)}.
\]
By the continuity of eigenvalues, we can now take $\epsilon_0$ positive
but sufficiently small so that, for all $p \in [ 1 : m - 1 ]$,
\begin{equation}\label{EBpeps}
B'_p + \epsilon_0 \bone_{( N - 1 ) \times ( N - 1 )} \in
\cS_{N - 1}^{(k_p)}\bigl( ( 0, \rho) \bigr) \quad \textrm{and} \quad 
B^{(m)}_0 + \epsilon_0 \bone_{N \times N} \in \cS_N^{(k_m)}
\bigl( ( 0, \rho) \bigr).
\end{equation}
After this preamble, we can now define the test matrices we need.
We take the partition
$\pi = \bigl\{ \{ 1, 2 \}, \{ 3 \}, \ldots, \{ N \} \bigr\}$
of $[ 1 : N ]$ into $N - 1$ subsets and then use the inflation operator
$\up$ from Definition~\ref{Dinflation} to produce
\begin{equation}\label{Emultitest2}
B^{(p)} := \left\{\begin{array}{ll}
\up(B'_p + \epsilon_0 \bone_{( N - 1 ) \times ( N - 1 )} ) & \textrm{if } p \in [ 1 : m - 1 ], \\[1ex]
B^{(m)}_0 + \epsilon_0 \bone_{N \times N} & \textrm{if } p = m.
\end{array}\right.
\end{equation}
By Lemma~\ref{Linertia}, we have that
$B^{(p)} \in \cS_N^{k_p}\bigl( ( 0, \rho ) \bigr)$
for all $p \in [ 1 : m ]$. Hence, by hypothesis, the matrix
$f[ \bB ] \in \cS_N^{(0)}$. In particular, its leading principal $2
\times 2$ submatrix is positive semidefinite, but this submatrix $M$
equals
\[
\begin{bmatrix} F( a ) & F( b ) \\ F( b ) & F( a ) \end{bmatrix}, \quad \textrm{where } %
F( x_m)  := f( \eps + \epsilon_0, \ldots, \eps + \epsilon_0, x_m ) %
\text{ for all } x_m \in ( 0, \rho ).
\]

As $f$ is absolutely monotone and we have assumed that its power-series
representation contains a monomial involving $x_m$, the function $F$ is
strictly increasing. This implies that the matrix~$M$ has negative
determinant, contradicting the fact that it is positive semidefinite.
Hence $f$ cannot depend on $x_m$, and similarly not on $x_p$ for any
$p > m_0$. Thus $(2)$ implies $(3)$.

Finally, suppose $\bk = \bzero$ and $l \geq 1$. The forward implication
is a special case of Proposition~\ref{Pabsmon}(1). For the converse, note
that if $g : \bx \mapsto f( \bx ) - f( \bzero )$ is absolutely monotone
then the entrywise transform $g[-]$ sends $\cS_n^{(\bzero)}(I)$ to
$\cS_n^{(0)}$, by the Schur product theorem.
Hence, by Lemma~\ref{LWeyl}, the entrywise map
$f[-]$ sends $\cS_n^{(\bk)}(I)$ to
$\overline{\cS_n^{(1)}} \subseteq \overline{\cS_n^{(l)}}$, as claimed.
\end{proof}

The remainder of this section is occupied by the following proof.

\begin{proof}[Proof of Theorem~\ref{Tmulti2} for $I = ( -\rho, \rho )$]
We first show that (3) implies (2).
Suppose $f$ has the form (\ref{Epresrep}) and let
$G( \bx' ) = F( \bx' ) - F( \bzero )$ for all
$\bx' := (x_1, \ldots, x_{m_0})$.
If the matrix $B^{(p)} \in \cS_n^{(k_p)}(I)$ for all $p \in [ 1 : m ]$
then
\[
f[ \bB] = F( \bzero) \bone_{n \times n} + G[ \bB_{[ 1 : m_0 ]} ]
+ \sum_{p = m_0 + 1}^m c_p B^{(p)}.
\]
Since $G$ is absolutely monotone and
$B^{(1)}$, $\ldots,$ $B^{(m_0)} \in \cS_n^{(0)}$,
the matrix $G[ \bB_{[ 1 : m_0 ]} ]$ is positive semidefinite.
Hence, by repeated applications of Lemma~\ref{LWeyl}, the
number of negative eigenvalues that the matrix $f[ \bB]$ has
is bounded above by that number for
\[
M := F( \bzero ) \bone_{n \times n} + %
\sum_{p = m_0 + 1}^m c_p B^{(p)}.
\]
By spectral decomposition, we can write $B^{(p)} = B^{(p)}_+ - B^{(p)}_-$,
where $B^{(p)}_+$ and $B^{(p)}_-$ are each the sum of rank-one positive
semidefinite matrices. Repeated application of Lemma~\ref{LWeyl} shows
that $B^{(p)} + B^{(q)} \in \overline{\cS_n^{(k_p + k_q)}}$ and therefore
we have that $M \in \overline{\cS_n^{(K)}}$, where
$K := \bone_{F( 0 ) < 0} + \sum_{p : c_p > 0} k_p \leq l$ by~(3)(c).
Hence $M$ has at most $l$ negative eigenvalues and therefore
so does $f[ \bB ]$.

As in the single-variable case, we now claim that this working extends to 
show that $(3) \implies (1)$.
Suppose $\bk' = ( k_1', \ldots, k_m' ) \in \Z_+^m \setminus \{ \bzero \}$
is such that $k'_j \leq k_j$ for all $j \in [ 1 : m ]$,
so that $m_0' \geq m_0$ and $k_{\max}' \leq k_{\max}$.
If (3) holds for $\bk$ then we can write
\[
f( \bx ) = F_0( x_1, \ldots, x_{m_0'} ) + \sum_{p = m_0' + 1}^m c_p' x_p,
\]
where
\[
F_0( x_1, \ldots, x_{m_0'} ) := %
F( x_1, \ldots, x_{m_0} ) + \sum_{p = m_0 + 1}^{m_0'} c_p x_p
\]
and $c_p' = c_p$ for $p \in [ m_0' + 1 : m ]$.
Since $F_0( \bzero_{m_0'} ) = F( \bzero_{m_0} )$, we have that
\[
\bone_{F_0( \bzero ) < 0} + \sum_{p : c_p' > 0} k_p \leq %
\bone_{F( \bzero ) < 0} + \sum_{p : c_p > 0} k_p \leq l,
\]
and therefore $f[ - ]$ maps $\cS_n^{(\bk')}( I )$ into
$\overline{\cS_n^{(l)}}$ for all $n \geq k_{\max}'$,
by the previous working.

That $(1) \implies (2)$ is immediate; finally, we show that (2) implies (3),
in several steps. We commence by noting that
Proposition~\ref{Pabsmon}(1) gives that
$g : \bx \mapsto f( \bx ) - f( \bzero )$
has a power-series representation with non-negative coefficients.

\noindent \textit{Step 1:
The only powers of $x_p$ with $p \in [ m_0 + 1 : m ]$ 
that may occur in any monomial in the power-series representation of
$f$ are $x_p^0$ and $x_p^1$.}

Lemma~\ref{LWeyl} implies that the transform
$g [ - ]$ sends $\cS_n^{(\bk)}(I)$ to $\overline{\cS_n^{(l + 1)}}$
for all $n \geq k_{\max}$.
We show the claim by applying Theorem~\ref{Tstrongest} to the restriction of $g$
to each coordinate~$x_{p_0}$, where $p_0 \in [ m_0 + 1 : m ]$. 
Given $\eps \in ( 0, \rho )$ and $B \in \cS_n^{(k_{p_0})}(I)$
for each $p \in [ 1 : m ]$, we define the block test matrix
\[
B^{(p)} := \left\{\begin{array}{ll}
B \oplus \bzero_{k_{\max} \times k_{\max}} & \text{if } p = p_0,\\[1ex]
\eps \bone_{n \times n} \oplus ( -\eps \Id_{k_p} ) \oplus
\bzero_{(k_{\max} - k_p) \times (k_{\max} - k_p)} \qquad &
\text{otherwise},\end{array}\right.
\]
where the term $-\eps \Id_{k_p}$ is omitted if $k_p = 0$.
As the block diagonal matrix $g[\bB]$ has at most $l+1$ negative
eigenvalues, so does its leading $n \times n$ block
$g_{p_0}[B]$, where
\[
g_{p_0} : I \to \R; \ %
x \mapsto g( \eps \bone_{p_0 - 1}^T, x, \eps \bone_{m - p_0}^T ).
\]
Thus $g_{p_0}[ - ]$ sends $\cS_n^{(k_{p_0})}(I)$ into
$\overline{\cS_n^{(l + 1)}}$, and it follows from Theorem~\ref{Tstrongest}
that $g_{p_0}$ is of the form required.

\noindent\textit{Step 2:
The function $f$ has the power-series representation
\[
f( \bx ) = F( \bx' ) + \sum_{p = m_0 + 1}^m c_p( \bx' ) x_p \qquad \text{for all } \bx \in I^m,
\]
where $\bx' = ( x_1, \ldots, x_{m_0} )$ and the functions $c_p$ and
$\bx' \mapsto F( \bx' ) - F( \bzero )$ are absolutely monotone.}

Fix any $\eps \in ( 0, \rho )$ and define
\[
h : I \to \R; \ x \mapsto g( \eps \bone_{m_0}^T, x \bone_{m - m_0}^T ).
\]
Then $h$ is a polynomial with non-negative coefficients whose degree is
independent of~$\eps$, by Step~1. The claim follows if we can show that $h$ is
linear.

We suppose for contradiction that the claim fails to hold, so that $d$, the degree of $h$,
is at least $2$. We suppose $N > ( d + 1 )^{l + 3}$ and let $\bu$ be as in (\ref{EdNu}),
where $u_0 \in ( 0, 1 )$, and $B$ be as in (\ref{EBs}), where
$\eps \in ( 0, \rho )$, $\epsilon$ is positive and  small enough to ensure $B$
has entries in $( 0, \rho )$ and $s_j := ( d + 1 )^j$ for $j \in [ 0 : l + 2 ]$.
For $p \in [ 1 : m ]$, we define the block test matrix
\begin{equation}\label{EStep2}
B^{(p)} := \left\{\begin{array}{ll}
\eps \bone_{N \times N} \oplus \bzero_{( k_{\max} - 1 ) \times ( k_{\max} - 1 )} & %
\text{if } p \in [1 : m_0],\\[1ex]
B \oplus ( -\eps \Id_{k_p - 1} ) \oplus %
\bzero_{( k_{\max} - k_p ) \times ( k_{\max} - k_p )} & \text{if } p \in [ m_0 + 1 : m ].
\end{array}\right.
\end{equation}
Since each $B^{(p)} \in \cS_{N + k_{\max} - 1}^{(k_p)}( I )$,
the matrix $g[ \bB ]$ has at most $l + 1$ negative eigenvalues.
Hence so does its leading $N \times N$ block,
which is precisely $h[ B ]$, but this is impossible by the
analysis following~(\ref{EBs}). This contradiction shows that $h$ is linear.

\noindent\textit{Step 3:
In the notation of Step~2, each absolutely monotone function $c_p$ is constant.}

To establish this claim we fix $\eps \in ( 0, \rho )$ and let
\[
A^{(p)} := \left\{\begin{array}{ll}
\eps \Id_{l + 1} \oplus \bzero_{k_{\max} \times k_{\max}} & %
\text{if } p \in [ 1 : m_0 ],\\[1ex]
\bzero_{( l + 1 + k_{\max} ) \times ( l + 1 + k_{\max}) } & \text{if } p \in [ m_0 + 1 : m ]
\end{array}\right.
\]
and
\[
B^{(p)} := \left\{\begin{array}{ll}
A^{(p)} & \text{if } p \in [ 1 : m_0 ],\\[1ex]
\eps \bone_{( l + 1 ) \times ( l + 1 )} \oplus \eps \Id_{k_p - 1} \oplus %
\bzero_{( k_{\max} - k_p + 1 ) \times ( k_{\max} - k_p + 1 )} & %
\text{if } p \in [ m_0 + 1 : m ].
\end{array}\right.
\]
Given any $\epsilon \in ( 0, \rho - \eps )$, we
let $\epsilon_1 = \cdots = \epsilon_m = \epsilon$
and apply Proposition~\ref{Pabsmon}(2)
to see that $f[ \bB + \bepsilon \bone ] - f[ \bA + \bepsilon \bone ]$
is positive semidefinite and has rank at most $l$. As this is a
block-diagonal matrix, the same holds for its
$( l + 1) \times ( l + 1 )$ leading principal submatrix~$M$.
In particular, $M$ is singular.

We now obtain a contradiction if any absolutely monotone function $c_p$
is non-constant. We have that
\begin{align*}
M & = \eps \sum_{p = m_0 + 1}^m %
c_p[ \eps \Id_{l + 1} + \epsilon \bone_{( l + 1 ) \times ( l + 1 )}, \ldots, %
\eps \Id_{l + 1} + \epsilon \bone_{( l + 1 ) \times ( l + 1 )} ]\\[1ex]
 & = %
 \bigl( h( \eps + \epsilon ) - h( \epsilon ) \bigr) \Id_{l + 1} + %
h( \epsilon ) \bone_{( l + 1 ) \times ( l + 1 )},
\end{align*}
where $h : x \mapsto \eps \sum_{p = m_0 + 1}^m c_p( x \bone_{m_0}^T )$.
If any function $c_p$ contains a non-trivial monomial, then so
does the absolutely monotone map $h$, which implies that
\[
a := h( \eps + \epsilon ) > b := h( \epsilon ) \geq 0,
\]
and so the matrix $M = ( a - b ) \Id_{l + 1} + b \bone_{( l + 1) \times ( l + 1 )}$ would be
positive definite.

\noindent \textit{Step 4:
Completing the proof.}

To conclude, we assume $f$ has the form (\ref{Epresrep2}) and show the
bound
\[
\bone_{F( \bzero ) < 0} + \sum_{p : c_p > 0} k_p \leq l.
\]

We adapt the test matrix in~(\ref{Ectrex}) and the subsequent discussion
to the present situation. We define $P_+ := \{ p \in [ m_0 + 1 : m ]
: c_p > 0 \}$; if $P_+$ is empty
then there is nothing to prove, so we assume otherwise. We let
$\{ v_1 := \bone_{K + 1}, v_2, \dots, v_{K + 1} \}$ be an orthogonal basis
of~$\R^{K + 1}$, where  $K := \sum_{p \in P_+} k_p$ and fix a partition
$\{ J_p : p \in P_+ \}$ of $[ 2 : K + 1 ]$ such that $|J_p| = k_p$ for all $p$.
We fix $\delta \in \bigl( 0, \rho \min\{ c_p : p \in P_+ \} \bigr)$ and choose a
positive~$\epsilon$ small enough to ensure the entries of the matrix
\begin{equation}\label{EAp}
A^{(p)} := \frac{\delta}{c_p |P_+|} \bone_{K + 1} \bone_{K + 1}^T -
\frac{\epsilon}{c_p} \sum_{j \in J_p} v_j v_j^T
\end{equation}
lie in $( 0 ,\rho )$. Lemma~\ref{Lsylv} then gives that
$A^{(p)} \in \cS_{K + 1}^{(k_p)}(I)$ for all $p \in P_+$.

We now let $N := K + k_{\max}$, choose $\eps \in ( 0, \rho )$, and use the
inflation operator $\up$ from Definition~\ref{Dinflation} to obtain the
test matrix
\[
B^{(p)} := \left\{\begin{array}{ll}
\bzero_{N \times N} & \text{if } p \in [ 1 : m_0 ],\\[1ex]
\up( A^{(p)} ) & \text{if } p \in P_+,\\[1ex]
\bzero_{( N - k_p ) \times ( N - k_p )} \oplus ( -\eps \Id_{k_p} ) & \text{otherwise},
\end{array}\right.
\]
where 
$\pi := \bigl\{ \{ 1 \}, \ldots, \{ K \}, \{ K + 1, \ldots, N \} \bigr\}$.
The matrix $B^{(p)} \in \cS_N^{(k_p)}( I )$ for all $p \in P_+$, by
Lemma~\ref{Linertia}, so $\bB \in \cS_N^{(\bk)}$ and therefore $f[ \bB ]
\in \overline{\cS_N^{(l)}}$. As we see directly that
\[
f[ \bB ] = %
F( \bzero ) \bone_{N \times N} + \sum_{p \in P_+} c_p B^{(p)} = %
\up\Bigl( F( \bzero ) \bone_{(K+1) \times (K+1)} + \sum_{p \in P_+} c_p A^{(p)} \Bigr),
\]
it follows from Lemma~\ref{Linertia} that
\begin{equation}\label{Efinal}
F( \bzero ) \bone_{( K + 1 ) \times ( K + 1 )} + \sum_{p \in P_+} c_p A^{(p)} = %
( F( \bzero ) + \delta ) \bone_{K + 1} \bone_{K + 1}^T - %
\epsilon \sum_{j = 2}^{K + 1} v_j v_j^T
\end{equation}
also has at most $l$ negative eigenvalues.

There are now two cases to consider. If $F( \bzero ) \geq 0$ then
Lemma~\ref{Lsylv} applied to~(\ref{Efinal}) gives that $K \leq l$, as
desired. If, instead, we have that $F( \bzero ) < 0$ then we may shrink
$\delta$ if necessary to ensure that $F( 0 ) + \delta < 0$, and then
again Lemma~\ref{Lsylv} gives the desired inequality, now $K + 1\leq l$.
\end{proof}

\section{Multi-variable transforms for matrices with positive entries}\label{S1sided}

In this section, we provide proofs for results stated above where
matrices have entries in the one-sided intervals $I = ( 0, \rho )$ and
$I = [ 0, \rho )$, where $0 < \rho \leq \infty$, and functions have
domains of the form $( 0, \rho )^m$ and $[ 0, \rho )^m$, where $m \geq 1$.

\begin{remark}
As with Remark~\ref{Rrescale}, such classification results imply
their counterparts for functions with domains of the form
\[
( 0, \rho_1 ) \times \cdots \times ( 0, \rho_m ) \qquad \text{or} \qquad
[ 0, \rho_1 ) \times \cdots \times [ 0, \rho_m ),
\]
where $0 < \rho_1, \ldots, \rho_m \leq \infty$.
\end{remark}

Having explored the one-variable and multi-variable situations separately
in the two-sided case, where $I = ( -\rho, \rho )$, in this section
we adopt a unified approach.

The case of $I = [ 0, \rho )$ will follow from the other two cases.
Thus, we first let $I = (0,\rho)$ and
defer the proofs for $I = [0,\rho)$ to Section~\ref{Snonneg}.
While the classification results for $I = ( -\rho, \rho )$ above
hold verbatim (except for the change of domain) when $I = ( 0, \rho )$,
the proofs need to be modified in several places. We describe these modifications
in what follows, and skip lightly over the remaining arguments,
which are essentially the same as those in Sections~\ref{S2} and~\ref{Smulti}.

The first step is to prove Proposition~\ref{Pabsmon} when $I = ( 0, \rho )$.
For this, we require the enhanced set of test matrices given by the following lemma.

\begin{lemma}\label{Lprelim2}
Given constants $a$ and $b$, with $0 \leq a < b$, a non-negative integer~$k$
and a non-negative constant $\epsilon$, the map
\[
\Psi = \Psi( a, b, k , \epsilon ) : \cS_n^{(0)} \to \cS_{k + 1 + n}^{(k)}; \ %
B \mapsto ( M_{k + 1}( a, b ) \oplus B ) + \epsilon \bone_{( k + 1 + n ) \times ( k + 1 + n )}
\]
is well defined, where
\[
M_{k + 1}( a, b ) := \begin{bmatrix}
 a & b & b & \cdots & b \\
 b & a & b & \cdots & b \\
 b & b & a & \cdots & b \\
 \vdots & \vdots & \vdots & \ddots & \vdots \\
 b & b & b & \cdots & a
\end{bmatrix} = %
( a - b ) \Id_{k + 1} + b \bone_{( k + 1 ) \times ( k + 1 )}.
\]
Furthermore, the $k$ negative eigenvalues of $\Psi( B )$ are all equal to
$a - b$ for any choice of $B$.
\end{lemma}

\begin{proof}
The $k=0$ case is immediate, so we assume henceforth that $k \geq 1$.

The matrix $b \bone_{( k + 1 ) \times ( k + 1 )}$ has rank one and
eigenvector $\bone_{k + 1}$ with eigenvalue $( k + 1 ) b$, so
$M_{k + 1}( a, b )$ has eigenvalue $a + k b$ with multiplicity $1$ and
eigenvalue $a - b$ with multiplicity $k$.

Let $P = ( k + 1)^{-1} \bone_{k + 1} \bone_{k + 1}^T$ and
$P^\perp = \Id_{k + 1} - P$,
so that $P$ and $P^\perp$ are spectral projections such that
$M_{k + 1}( a, b ) = ( a + k b ) P + ( a - b ) P^\perp$ and
\[
\Psi( B ) = \bigl( ( a + k b ) P + ( a - b ) P^\perp \bigr) \oplus B + %
\epsilon \bone_{(k + 1 + n ) \times ( k + 1 + n )}.
\]
If $\bv$ lies in the range of $P^\perp$, which has dimension $k$,
then $P \bv = \bzero_{k + 1}$ and $\bone_{k + 1}^T \bv = 0$, so
\[
\Psi( B ) \begin{bmatrix} \bv \\ \bzero_n \end{bmatrix} = %
( a - b ) \begin{bmatrix} \bv \\ \bzero_n \end{bmatrix}.
\]
This shows that $a - b$ is an eigenvalue of $\Psi( B )$ with multiplicity
at least $k$. As the matrix $M_{k + 1}( a, b ) \oplus B$ has $k + 1$
non-negative eigenvalues, so does $\Psi( B )$, by
\cite[Corollary~4.3.9]{HJ} and the previous observation. This completes
the proof.
\end{proof}

\begin{proof}[Proof of Proposition~\ref{Pabsmon} and
Corollary~\ref{Csample} for $I = (0,\rho)$]

The reader can verify that the above proofs of
Proposition~\ref{Pabsmon}(2) and Corollary~\ref{Csample}
with $I = ( -\rho, \rho )$ go through
verbatim for $I = ( 0,\rho )$, since all of the test matrices used therein
have all their entries in $( 0, \rho )$.
It remains to show that the first part of Proposition~\ref{Pabsmon}
holds in this setting.

We begin by fixing $\epsilon \in I = ( 0, \rho )$
and $a$, $b \in ( 0, \rho - \epsilon )$ with $a < b$.
Given a positive integer $n$ and an $m$-tuple of matrices
$\bA = ( A^{(1)}, \ldots, A^{(m)} ) \in %
\cS_n^{(\bzero)}\bigl( ( 0, \rho - \epsilon ) \bigr)$,
the matrix
\[
B_0^{(p)} := \Psi( a, b, k_p, 0)\bigl( ( A^{(p)} )^{\oplus ( l + 2 )} \bigr) = %
M_{k_p + 1}( a, b ) \oplus ( A^{(p)} )^{\oplus ( l + 2 ) }
\]
is an $N_p \times N_p$ block-diagonal matrix
with entries in $[ 0, \rho - \epsilon )$, where $N_p := k_p + 1 + ( l + 2 ) n$,
for all $p \in [ 1 : m ]$.
By Lemma~\ref{Lprelim2}, the matrix
$B_0^{(p)} + \epsilon \bone_{N_p \times N_p} \in \cS^{(k_p)}_{N_p}( I )$.
We apply the inflation operator $\up[p]$ to this matrix, where
\[
\pi_p := %
\bigl\{ \{ 1, \ldots, k_{\max} - k_p + 1 \}, \{ k_{\max} - k_p + 2 \}, %
\ldots, \{ k_{\max} - k_p + N_p \} \bigr\}.
\]
By Lemma~\ref{Linertia}, the matrix
$B^{(p)} :=  \up[p]( B_0^{(p)} )$
is such that $B^{(p)} + \epsilon \bone_{N \times N} \in \cS_N^{(k_p)}( I )$,
where $N := k_{\max} + 1 + ( l + 2 ) n$.

It now follows from the hypotheses of the theorem that
$f[ \bB + \bepsilon \bone ] \in \overline{\cS_N^{(l)}}$,
where $\epsilon_1 = \cdots = \epsilon_m = \epsilon$. Thus,
by Lemma~\ref{LWeyl}, we have that
\[
g_\epsilon[ \bB ] \in \overline{\cS_N^{(l + 1)}}, %
\qquad \textrm{where } %
g_\epsilon( \bx ) := %
f( \bx + \epsilon \bone_m ) - f( \epsilon \bone_m ).
\]
The matrix $g_\epsilon[ \bB ]$ is block diagonal of the form
\[
g_\epsilon[ \up[1']\bigl( M_{k_1 + 1}( a, b ) \bigr), \ldots,
\up[m']\bigl( M_{k_m+ 1 }( a, b ) \bigr) ] \oplus g_\epsilon[ \bA ]^{\oplus ( l + 2 )},
\]
where
\[
\pi_{p'} := \{ A \cap \{ 1, \ldots, k_{\max} + 1 \} : A \in \pi_p \} \qquad \textrm{for all } p \in [ 1 : m ].
\]
Thus $g_\epsilon[ \bA ]$ cannot have any negative eigenvalues,
for any $m$-tuple $\bA \in \cS_n^{( \bzero )}\bigl( ( 0, \rho - \epsilon ) \bigr)$
and all $n \geq 1$.
It now follows from Theorem~\ref{T000} that
\[
f( \bx + \epsilon \bone_m ) - f( \epsilon \bone_m ) = g_\epsilon( \bx ) = %
\sum_{\balpha \in \Z_+^m} c_{\balpha, \epsilon} \bx^\balpha %
\quad \textrm{for all } \bx \in ( 0, \rho - \epsilon )^m,
\]
where $c_{\balpha, \epsilon} \geq 0 \textrm{ for all } \balpha \in \Z_+^m$.

As this holds for all $\epsilon \in ( 0, \rho )$, we see that $f$ is
smooth on~$( 0, \rho )^m$ and $( \partial^\balpha f ) ( \bx ) \geq 0$ for
all $\balpha \neq \bzero$ and all $\bx \in ( 0, \rho )^m$. In particular,
$\partial_{x_p} f$ is absolutely monotone on $( 0, \rho )^m$ for any
$p \in [ 1: m ]$ and so, by Theorem~\ref{Tabsmon}, it has there a
power-series representation with non-negative Maclaurin coefficients: we
can write
\[
( \partial_{x_p} f )( \bx ) = %
\sum_{\balpha \in \Z_+^m} ( \alpha_p + 1 ) c^{(p)}_{\balpha + \be_p^T} \bx^\balpha %
\qquad \textrm{for all } \bx \in ( 0, \rho )^m,
\]
where $c^{(p)}_{\balpha + \be_p^T} \geq 0$ for all $\balpha$.
Since $f$ is smooth, the mixed partial derivatives
$\partial_{x_p} \partial_{x_q} f$ and $\partial_{x_q} \partial_{x_p} f$
are equal for any distinct $p$ and $q$, whence
$c^{(p)}_\balpha = c^{(q)}_\balpha$ whenever $\alpha_p > 0$
and~$\alpha_q > 0$. Hence setting $c_\balpha := c^{(p)}_\balpha$
makes $c_\balpha$ well defined whenever $\balpha \neq \bzero$. We let
\[
\wt{f}( \bx ) := \sum_{\balpha \neq \bzero} c_\balpha \bx^\balpha \qquad %
\textrm{for all } \bx \in ( 0, \rho )^m,
\]
which is convergent because
\[
\sum_{\balpha \neq \bzero} | c_\balpha \bx^{\balpha} | \leq %
\sum_{p = 1}^m x_p %
\sum_{\balpha : \alpha_p > 0} ( \alpha_p + 1 ) c_{\balpha + \be_p^T} \bx^\alpha.
\]
Note that $\partial_{x_p} \wt{f} = \partial_{x_p} f$ on $( 0, \rho )^m$
for any $[ 1 : m ]$ and therefore $f = \wt{f} + c_\bzero$ if we let
$c_\bzero := f( \epsilon \bone_m ) - \wt{f}( \epsilon \bone_m )$.
This shows that $x \mapsto f( \bx )$ has a power-series representation
on~$( 0, \rho )^m$ of the form claimed and so completes the proof of
Proposition~\ref{Pabsmon}(1) when $I = ( 0, \rho )$.
\end{proof}

With Proposition~\ref{Pabsmon} and Corollary~\ref{Csample} now
established for $I = ( 0, \rho )$, we next show that the two main
theorems hold in this context.

We note first that if $f$ has a power-series representation on $( 0, \rho )^m$
as in the conclusion of Proposition~\ref{Pabsmon}(1) then the unique
extension of the function $f$ to $( -\rho, \rho )^m$ will also be denoted
by $f$ and will be used without further comment.

\begin{proof}[Proof of Theorem~\ref{Tto0} for $I = (0,\rho)$]
The implication $(1) \implies (2)$ is immediate, and that $(3) \implies (1)$
follows from the Schur product theorem. It remains to show that $(2) \implies (3)$
and also the case when $\bk = \bzero$ and $l \geq 1$. Given~(2), we note that
Proposition~\ref{Pabsmon}(1) shows that $f$ has a
power-series representation
\begin{equation}\label{Efpowerseries}
f( \bx ) = \sum_{\balpha \in \Z_+^m} c_\balpha \bx^\balpha \quad %
\text{for all } \bx \in (0, \rho)^m
\end{equation}
with the coefficients $c_\balpha \geq 0$ for all $\balpha \neq \bzero$.
In particular, the function $\bx \mapsto f( \bx ) - c_\bzero$ is continuous, non-negative,
and non-decreasing on $( 0, \rho )^m$ and hence extends to a continuous function on
$[ 0, \rho )^m$. We denote this extension by $g$ and use $g$ to extend $f$ to
$[ 0, \rho )^m$ so
that $f( \bx ) = g( \bx ) + c_\bzero$ for all $\bx \in [0,\rho)^m$.

Now the proof of Theorem~\ref{Tto0} for $I = ( -\rho, \rho )$ goes
through verbatim, with one exception: the test matrix
in~(\ref{Etestmatrix3}) must be changed. We let $N := k_{\max} + 1$ and
\[
\pi_p := \bigl\{ \{ 1 \}, \{ 2 \}, \ldots, \{ k_p \},
\{ k_p+1, \ldots, N \} \bigr\}.
\]
Given $a$, $b \in ( 0, \rho )$, with $a < b$, the matrix
\[
B^{(p)} := \up[p]\bigl( M_{k_p + 1}( a, b ) \bigr) \in \cS_N^{(k_p)}( I ) %
\qquad \textrm{for all } p \in [ 1 : m ],
\]
by Lemma~\ref{Lprelim2} and Lemma~\ref{Linertia}.
It follows that the matrix $f[ \bB ]$ is positive semidefinite, and so
its $( 1, 1 )$ entry $f( a \bone_m ) \geq 0$.
Since $f$ is continuous on $[ 0, \rho )^m$,
letting $a \to 0^+$ gives that $f( \bzero ) \geq 0$, as required.
\end{proof}

\begin{proof}[Proof of Theorem~\ref{Tmulti2} for $I = (0,\rho)$]
That $(1) \implies (2)$ is immediate, while the proofs that $(3) \implies (2)$
and $(3) \implies (1)$ go through as for the case where $I = (-\rho,\rho)$.
To show $(2) \implies (3)$, we note as above that
Proposition~\ref{Pabsmon}(1) and the remarks immediately
after~(\ref{Efpowerseries}) show the function $f$ extends
continuously to the domain $[ 0, \rho )^m$ with
the power-series representation~(\ref{Efpowerseries}) there, and
$g : \bx \mapsto f( \bx ) - f( \bzero )$ is absolutely monotone on this
domain. We will proceed parallel to the proof of Theorem~\ref{Tmulti2},
indicating the changes required for each of its steps.

The first step is somewhat different and its proof is entirely changed.

\noindent \textit{Step $1'$:
For each $p \in [ m_0 + 1 : m ]$, only finitely many powers of
$x_p$ occur in any monomial in the power-series representation of $f$.}

Here we apply Corollary~\ref{Csample} instead of Theorem~\ref{Tstrongest}
to the restriction of $g$ to each coordinate $x_{p_0}$.
To do this, we fix $p_0 \in [ m_0 + 1 : m ]$ and 
$n \geq \max\{ k_{p_0}, l + 3 \}$, and suppose $B \in \cS_n^{(0)}\bigl( [ 0, \rho ) \bigr)$
has rank $k_{p_0}$. We take $\epsilon > 0$ such that
$B + \epsilon \bone_{n \times n}$ has all its entries in $( 0, \rho )$,
choose $\eps \in ( 0, \rho - \epsilon )$, and let $N := n + k_{\max}$.
We let the matrix
\[
B^{(p)} := \left\{\begin{array}{ll}
B \oplus \bzero_{k_{\max} \times k_{\max}} & \text{if } p = p_0,\\[1ex]
\eps \bone_{n \times n} \oplus \eps \Id_{k_p - 1} \oplus %
\bzero_{( k_{\max} - k_p + 1 ) \times ( k_{\max} - k_p + 1 )} & %
\text{otherwise}.
\end{array}\right.
\]
We can write $g$ in the split representation
\[
g( \bx ) = \sum_{n = 0}^\infty c_n( \bx_{[ 1 : m ] \setminus \{ p_0 \}}) x_{p_0}^n
\quad \text{for all } \bx \in [ 0, \rho )^m
\]
and so the function $g_{p_0}' : [ 0, \rho ) \to \R$ such that
\[
g_{p_0}'( x ) := %
g( \epsilon \bone_{p_0 - 1}^T, x, \epsilon \bone_{m - p_0}^T ) - %
g( \epsilon \bone_{p_0 - 1}^T, 0, \epsilon \bone_{m - p_0}^T ) = %
\sum_{n = 1}^\infty c_n( \epsilon \bone_{m - 1}^T ) x^n.
\]
Applying Corollary~\ref{Csample}(2) with $P = \{ p_0 \}$
and $\epsilon_p := \epsilon$ for all $p \in [ 1: m ]$, the matrix
\[
g_{p_0}'[B^{(p_0)}] = %
\sum_{n=1}^\infty c_n( \epsilon \bone_{m-1}^T ) (B^{(p_0)})^{\circ n}
\]
is positive semidefinite and has rank at most $l$. As $g_{p_0}'( 0 ) = 0$,
this is a block matrix and so its leading $n \times n$ block has the same
property. Hence $g_{p_0}'$ sends any
$B \in \cS_n^{(0)}\bigl( [ 0, \rho ) \bigr)$ with rank $k_{p_0}$ to a
positive semidefinite matrix of rank at most $l$.
As $n \geq \max\{ k_{p_0}, l + 3 \}$, applying
\cite[Theorems~A and~B]{GKR-lowrank} with $I = [ 0, \rho )$ gives that
$g_{p_0}'$ is a polynomial. This proves the claim.

\noindent\textit{Step 2:
The function $f$ has the power-series representation
\[
f( \bx ) = F( \bx' ) + \sum_{p = m_0 + 1}^m c_p( \bx' ) x_p \qquad \text{for all } \bx \in I^m,
\]
where $\bx' = ( x_1, \ldots, x_{m_0} )$ and the functions $c_p$ and
$\bx' \mapsto F( \bx' ) - F( \bzero )$ are absolutely monotone.}

The proof here is the same as above for the case when $I = ( -\rho, \rho )$  until
the definition of the test matrices in~(\ref{EStep2}). Instead we let
$a$, $b \in ( 0, \rho )$ be such that $a < b$ and take
\[
B^{(p)} := \left\{\begin{array}{ll}
\eps \bone_{N \times N} \oplus \bzero_{k_{\max} \times k_{\max}} & %
\text{if } p \in [1 : m_0],\\[1ex]
B \oplus M_{k_p}( a, b ) \oplus \bzero_{( k_{\max} - k_p ) \times ( k_{\max} - k_p )} & %
\text{if } p \in [m_0 + 1 : m],
\end{array}\right.
\]
where $M_{k_p}( a, b )$ is as in Lemma~\ref{Lprelim2}.
The matrix $B^{(p)} \in \cS_{N + k_{\max}}^{(k_p)}\bigl( [ 0,  \rho ) \bigr)$
and therefore we have that
$B^{(p)} + \epsilon \bone_{( N + k_{\max} ) \times ( N + k_{\max} )} \in %
\cS_{N + k_{\max}}^{(k_p)}( I )$
for all $p$ and all sufficiently small and positive $\epsilon$,
by the continuity of eigenvalues. Hence
$g[ \bB + \bepsilon \bone_{( N + k_{\max} ) \times ( N + k_{\max} )} ]$
has at most $l+1$ negative eigenvalues, where $\epsilon_p = \epsilon$ for all
$p$. 
Letting $\epsilon \to 0^+$, the same holds for $g[ \bB ]$, again by the
continuity of eigenvalues, since $\overline{\cS_{N+k_{\max}}^{(l+1)}}$
is closed for the topology of entrywise convergence.
As $g[ \bB ]$ is a block matrix, its leading $N \times N$ block $h[ B ]$ also
has at most $l+1$ negative eigenvalues, where we recall that
$h(x) := g( \eps \bone_{m_0}^T, x \bone_{m-m_0}^T )$.
As before, this is not possible according to
the analysis following~(\ref{EBs}) if $h$ has degree~$2$ or more. Thus $h$
must be linear, which proves the claim.

\noindent \textit{Step 3:
In the notation of Step~2, each absolutely monotone function $c_p$ is constant.}

The proof above of this step when $I = (-\rho, \rho)$ goes through verbatim if
$I = ( 0, \rho )$.

\noindent \textit{Step 4:
Completing the proof.}

The proof of this step with $I = ( -\rho, \rho )$ may be imitated verbatim,
until and including the definition in~(\ref{EAp}) of the matrix $A^{(p)}$ for $p \in P_+$.
To continue in the current setting, where $I = ( 0, \rho )$, we let $N := K + k_{\max} + 1$,
choose $a$, $b \in (0, \rho)$ with $a < b$, and take some $\epsilon \in ( 0, \rho - b )$.
With the help of the inflation operator $\up$ from Definition~\ref{Dinflation}, we let
\[
B^{(p)} := \left\{\begin{array}{ll}
b \bone_{N \times N} & \text{if } p \in [ 1 : m_0 ],\\[1ex]
\up( A^{(p)} ) & \text{if } p \in P_+,\\[1ex]
(\bzero_{( N - k_p - 1) \times ( N - k_p - 1)} \oplus M_{k_p + 1}( a, b )) + %
\epsilon \bone_{N \times N} \qquad & \text{otherwise},
\end{array}\right.
\]
where 
$\pi := \bigl\{ \{ 1 \}, \ldots, \{ K \}, \{ K + 1, \ldots, N \} \bigr\}$ and
$M_{k_p + 1}( a, b )$ is as in Lemma~\ref{Lprelim2}.
It follows from Lemma~\ref{Lprelim2} and Lemma~\ref{Linertia}
that $\bB \in \cS_N^{(\bk)} \bigl( (0, \rho) \bigr)$. Hence
\[
f[ \bB ] = F[ B^{(1)}, \ldots, B^{(m_0)} ] + %
\sum_{p \in P_+} c_p B^{(p)} \in \overline{\cS_N^{(l)}}.
\]
Since $F$ is continuous, letting $b \to 0^+$ gives that
\begin{align*}
F[ \bzero_{N \times N}, \ldots, \bzero_{N \times N} ] + %
\sum_{p \in P_+} c_p B^{(p)} & = %
F( \bzero_{m_0}^T ) \bone_{N \times N} + %
\sum_{p \in P_+} c_p B^{(p)}\\
& = \up\Bigl( F( \bzero ) \bone_{( K + 1 ) \times ( K + 1 )} + %
\sum_{p \in P_+} c_p A^{(p)} \Bigr) \in \overline{\cS_N^{(l)}},
\end{align*}
by the continuity of eigenvalues. Hence the matrix
in~(\ref{Efinal}) also has at most $l$ negative eigenvalues, by Lemma~\ref{Linertia},
and the rest of the proof following~(\ref{Efinal}) goes through unchanged.
\end{proof}

\begin{remark}
We could have used Step~$1'$ and its proof as above in place of Step~{1}
in the proof given in Section~\ref{Smulti} of Theorem~\ref{Tmulti2} for
$I = ( -\rho, \rho )$.
\end{remark}

We now use Theorem~\ref{Tmulti2} to establish two results for preservers
when $I = ( 0, \rho )$, including a version of
Corollary~\ref{Cpreserver}. The proofs require another carefully chosen
set of test matrices.

\begin{lemma}\label{Ljudicious2}
There exists an invertible matrix $A \in \cS_3^{(1)}\bigl( ( 0, 5 ) \bigr)$
such that, for each positive integer $k$, the matrix
\[
A^{\oplus k} + t \bone_{3 k \times 3 k} \in \cS_{3k}^{(k - 1)} %
\qquad \textrm{whenever $t > 1$}.
\]
\end{lemma}

\begin{proof}
Note first that the matrix
\[
A = \begin{bmatrix}
 4 & 2 & 3 \\ 2 & 1 & 2 \\ 3 & 2 & 4
\end{bmatrix}
\]
has characteristic polynomial
\[
\det( x \Id_3 - A ) = ( x - 1 ) ( x^2 - 8 x - 1 )
\]
and eigenvalues $1$, $4 + \sqrt{17}$ and $4 - \sqrt{17}$, whence
$A \in \cS_3^{(1)}\bigl( ( 0, 5 ) \bigr)$.

To proceed, we recall Cauchy's matrix determinant lemma, which can be found as equation (0.8.5.11)
on p.26 of \cite{HJ} and states that
\[
\det( B + \bu \bv^T ) = \det B + \bv^T (\adj B) \, \bu,
\]
where $\adj B$ denotes the adjugate of the matrix $B$.
We also note the identity
\[
\adj( B \oplus C ) = ( \det C ) \adj B \oplus ( \det B ) \adj C.
\]
It follow from these identities that,
for any $t \in \R$,
\begin{multline*}
\det( x \Id_{3 k} - A^{\oplus k} - t \bone_{3 k \times 3k } ) \\
= \det( x \Id_3 - A )^{k - 1} ( \det( x \Id_3 - A ) - %
t k \bone_3^T \adj( x \Id_3 - A ) \bone_3 ).
\end{multline*}
A straightforward computation reveals that
\[
\bone_3^T \adj( x \Id_3 - A ) \bone_3 = ( 3 x - 1 ) ( x - 1 )
\]
and so the characteristic polynomial of the linear pencil
$A^{\oplus k} + t \bone_{3 k \times 3 k}$ is
\[
\det( x \Id_3 - A )^{k - 1} ( x - 1 ) %
( x^2 - ( 8 + 3 t k ) x - 1 + t k ).
\]
Thus $A^{\oplus k} + t \bone_{3 k \times 3 k} \in \cS_{3 k}^{(k - 1)}$
if both roots of the polynomial $x^2 - ( 8 + 3 t k ) x - 1 + t k$ are
positive, but this holds for all $t > 1$.
\end{proof}

Now we have:

\begin{proof}[Proof of Corollary~\ref{Cpreserver} for $I = (0,\rho)$]
The proof with $I = ( -\rho, \rho )$ of the second part goes through
verbatim over $( 0, \rho )$, as does the proof of the first part until
the final sentence, which uses the matrices
$B^{(1)} = \cdots = B^{(m)} = -\epsilon \Id_k$ having no entries in
$(0,\rho)$. Thus, it remains to show the following holds.

\noindent \textit{Suppose $I := ( 0, \rho )$, where $0 < \rho \leq \infty$,
and let $f( \bx ) = c x_{p_0} + d$ for all $\bx \in I^m$, where $c > 0$,
$p_0 \in [ 1 : m ]$ and $d \geq 0$.
If $k \geq 1$ and
$f[ - ]$ sends $\cS_n^{(k \bone_m^T)}(I)$ to $\cS_n^{(k)}$
for all $n \geq k$ then $d=0$.}

To show this, suppose $d> 0$ and
let $A \in \cS_3^{(1)}\bigl( ( 0, 5 ) \bigr)$
be as in Lemma~\ref{Ljudicious2}. Choose $\delta \in ( 0, d / c )$
such that $\delta < \rho / 5$, so that
$( \delta A )^{\oplus k} \in \cS_{3k}^{(k)}\bigl( [ 0,\rho) \bigr)$.
This matrix is invertible, so
$( \delta A)^{\oplus k} + \epsilon \bone_{3 k \times 3 k} \in \cS_{3
k}^{(k)}\bigl( ( 0, \rho ) \bigr)$ for sufficiently small $\epsilon > 0$,
by the continuity of eigenvalues. Then, by the assumption on $f$,
\[
B := f[ ( \delta A )^{\oplus k} + \epsilon \bone_{3 k \times 3 k}, %
\ldots, ( \delta A )^{\oplus k} + \epsilon \bone_{3 k \times 3 k} ] = %
c ( \delta A )^{\oplus k} +  (c \epsilon + d ) \bone_{3 k \times 3 k} %
\in \cS_{3 k}^{(k)}.
\]
Furthermore, we can write
\[
B = c \delta ( A^{\oplus k} + ( d / c \delta ) \bone_{3 k \times 3 k} ) + %
c \epsilon  \bone_{3 k \times 3 k },
\]
and Lemma~\ref{Ljudicious2} implies that
$A^{\oplus k} + t \bone_{3 k \times 3 k} \in \cS_{3 k}^{(k - 1)}$
for all $t > 1$, so $B \in \overline{\cS_{3 k}^{(k - 1)}}$, by
Lemma~\ref{LWeyl}.

This contradiction shows that $d = 0$, as claimed.
\end{proof}

We end by returning full circle to the first result we stated
in the one-variable setting: the classification of
inertia preservers for matrices with positive or non-negative entries.

\begin{corollary}\label{Cinertia}
Let $I := ( 0, \rho )$ or $ [ 0, \rho )$, where $0 < \rho \leq \infty$,
and let $k$ be a non-negative integer.
Given a function $f : I \to \R$, the following are equivalent.
\begin{enumerate}
\item The entrywise transform $f[-]$ preserves the inertia of all
matrices in $\cS^{(k)}(I)$.

\item The function is a positive homothety: $f( x ) \equiv c x$
for some constant $c > 0$.
\end{enumerate}
\end{corollary}

In other words, Theorem~\ref{Tinertia} holds verbatim if
$I = ( -\rho, \rho )$ is replaced by $( 0, \rho )$ or~$[ 0, \rho )$.

\begin{proof}[Proof for $I = ( 0, \rho )$]
Clearly, (2) implies (1). Conversely, if $k = 0$ then the proof
of Theorem~\ref{Tinertia} goes through in this case, using
Theorem~\ref{T000} with $m = 1$ in place of
Schoenberg's Theorem~\ref{Tschoenberg}. Otherwise,
we have $k \geq 1$ and $f[-]$ sends
$\cS_n^{(k)}\bigl( ( 0, \rho ) \bigr)$ into $\overline{\cS_n^{(k)}}$
for all $n \geq k$. The
$m=1$ case of Theorem~\ref{Tmulti2} with $k_1 = l = k$ gives that
either $f( x ) \equiv d$ for some $d \in \R$ or $f( x ) \equiv c x + d$,
with $c > 0$ and $d \geq 0$. Now we are done by following the proof of
Theorem~\ref{Tinertia}
to out the first possibility and using the italicized
assertion in the proof of Corollary~\ref{Cpreserver} to show that $d = 0$.
\end{proof}

\subsection{Proofs for matrices with non-negative entries}\label{Snonneg}

We conclude by explaining how to modify the proofs given above
to obtain Theorems~\ref{Tto0} and~\ref{Tmulti2},
together with their classification consequences in
Corollaries~\ref{Cpreserver} and~\ref{Cinertia}, when $I = [ 0,
\rho )$. We begin by establishing Proposition~\ref{Pabsmon} and
Corollary~\ref{Csample} in this context.

\begin{proof}[Proof of Proposition~\ref{Pabsmon} and
Corollary~\ref{Csample} for $I = [0,\rho)$]
Proofs of the second parts of Proposition~\ref{Pabsmon}
and Corollary~\ref{Csample} when $I = ( 0, \rho )$
go through verbatim if $I = [ 0, \rho )$, since the
test matrices used there have all their entries in $( 0, \rho )$.
It remains to prove the first part of Proposition~\ref{Pabsmon}.
This follows the same reasoning as the proof of
Proposition~\ref{Pabsmon}(1) for $I = ( -\rho, \rho )$,
but with $B^{(p)}$ there replaced by
\[
B^{(p)} := M_{k_p + 1}( a, b ) \oplus %
\bzero_{( k_{\max} - k_p ) \times ( k_{\max} - k_p )} \oplus %
A^{\oplus( l + 2 )}_p \in \cS_N^{(k_p)}
\]
for each $p \in [ 1 : m ]$, where $a$, $b \in ( 0, \rho )$ with $a < b$
and $M_{k_p + 1}( a, b )$ is as in Lemma~\ref{Lprelim2}.
\end{proof}

With these results in hand, we can now provide the final proofs.

\begin{proof}[Proof of Theorems~\ref{Tto0} and~\ref{Tmulti2}, and of
Corollaries~\ref{Cpreserver} and~\ref{Cinertia}, for $I = [0,\rho)$]
First, it is clear that (1) implies (2) in Theorem~\ref{Tto0}. If (2) holds
then $f[-]$ sends $\cS_n^{(\bk)}\bigl( ( 0, \rho ) \bigr)$
to $\cS_n^{(0)}$, for all $n \geq k_{\max}$, so
the $I = ( 0, \rho )$ part of Theorem~\ref{Tto0} implies that
the restriction of $\bx \to f( \bx )$ to $( 0, \rho )^m$ is independent of
$x_{m_0 + 1}$, $\ldots,$ $x_m$.
However, Proposition~\ref{Pabsmon}(1) implies that
$f$ is continuous on $[ 0, \rho )^m$, so this independence
extends to the whole domain of $f$ and therefore (2) implies (3).
Finally, $(3) \implies (1)$ by the Schur product theorem.
The proof of the final part of Theorem~\ref{Tto0} is identical to the
$I = ( -\rho, \rho )$ version.

For Theorem~\ref{Tmulti2}, clearly $(1) \implies (2)$.
If (2) holds then (3) holds for the restriction of $f$ to
$( 0, \rho)^m$, and again Proposition~\ref{Pabsmon}(1) allows
extension by continuity, so that (2) holds in general. The proof that
(3) implies (1) is unchanged from the $I = ( -\rho, \rho )$
case.

For both parts of Corollary~\ref{Cpreserver}, if $f$ has the
prescribed form then the entrywise transform maps
$\cS_n^{(k \bone_m^T)}( I )$
or $\overline{\cS_n^{(k \bone_m^T)}}( I )$
to $\cS_n^{(k)}$ or $\overline{\cS_n^{(k)}}$
when $I = ( -\rho , \rho )$, so the same holds when $I = [ 0, \rho )$.
The converse uses the same restriction and continuity argument as
before.

Finally, that (2) implies (1) in Corollary~\ref{Cinertia}
is immediate. For the converse, note again that the
conclusion holds for the restriction of $f$ to $( 0, \rho )$ and the
continuity of $f$ on $I = [0,\rho)$ follows from
Theorem~\ref{Tmain} and Proposition~\ref{Pabsmon}(1).
\end{proof}

\section{Transforms of complex Hermitian matrices}\label{Scomplex}

This section is devoted to understanding the class of entrywise
negativity preservers for tuples of Hermitian matrices with complex
entries. We begin by recording the elementary observation that entrywise
conjugation preserves $\cS^{(k)}_n( \C )$ for all $k \in [ 0 : n ]$ and
all $n \geq 1$, and by establishing the following lemma.

\begin{lemma}\label{L2x2}
There exists a $2 \times 2$ complex Hermitian matrix $A$ which is
singular, has one positive eigenvalue, and is such that the matrix $c A +
d \overline{A}$ is positive definite for any choice of $c$, $d \in ( 0,
\infty )$.
\end{lemma}

\begin{proof}
Choose any complex number $\zeta$ of unit modulus other than $1$ and $-1$
and let
\[
A := \begin{pmatrix} 1 & \overline{\zeta} \\ \zeta & 1 \end{pmatrix}.
\]
This matrix is Hermitian and has zero determinant and positive trace.
If $c$, $d \in ( 0, \infty )$ then
\[
\frac{1}{c + d} ( c A + d \overline{A} )
= \begin{pmatrix} 1 & \overline{\eta} \\ \eta & 1 \end{pmatrix},
\]
where $\eta = ( c \zeta + d \overline{\zeta} ) / ( c + d )$ lies in the
interior of the line segment joining $\zeta$ and $\overline{\zeta}$, so
is in the open unit disc. The result follows.
\end{proof}

We are now equipped to prove the main result.

\begin{proof}[Proof of Theorem~\ref{Tmain-complex}]
We first assume that~(3) holds, so that the function $f$ has the
representation~(\ref{Epresrep3}), and show~(2), following the proof of
Theorem~\ref{Tmulti2} for $I = ( -\rho, \rho )$.
Given $B^{(p)} \in \cS_n^{(k_p)}(\C)$ for all $p \in [ 1 : m ]$, we have that
\[
f[ \bB ] = F( \bzero ) \bone_{n \times n} + G[ \bB_{[ 1 : m_0 ]} ]
+ \sum_{p = m_0 + 1}^m \bigl( c_p B^{(p)} + d_p \overline{B^{(p)}} \bigr),
\]
where $G( \bz' ) = F( \bz' ) - F( \bzero )$ for all $\bz' := ( z_1, \ldots, z_{m_0} )$.
Since $B^{(1)}$, \ldots, $B^{(m_0)} \in \cS_n^{(0)}( \C )$,
the matrix $G[ \bB_{[ 1 : m_0 ]} ]$ is positive semidefinite
by the Schur product theorem.
Hence, by repeated applications of Lemma~\ref{LWeyl}, which remains valid
for complex Hermitian matrices, the number of negative eigenvalues of
the matrix $f[ \bB]$ is bounded above by the number of
negative eigenvalues of
\[
M := F( \bzero ) \bone_{n \times n} + %
\sum_{p = m_0 + 1}^m \bigl( c_p B^{(p)} + d_p \overline{B^{(p)}} \bigr).
\]
As in the real case, we may use spectral decomposition and repeated
applications of Lemma~\ref{LWeyl} to see that $M \in \overline{\cS_n^{(K)}}( \C )$,
where $K := \bone_{F( 0 ) < 0} + \sum_{p : c_p > 0} k_p + \sum_{p : d_p > 0} k_p$.
We conclude from condition (3)(c) that $M$ has at most $l$ negative eigenvalues. Hence so
does $f[ \bB ]$.

This shows that $(3) \implies (2)$, and the same working as for the
real case now yields that $(3) \implies (1)$. Moreover,
that $(1) \implies (2)$ is trivial.

To conclude, we show that $(2) \implies (3)$. The proof again follows similar
steps to those proving Proposition~\ref{Pabsmon} and Theorem~\ref{Tmulti2}
when $I = ( -\rho, \rho )$, and also makes use of Theorem~\ref{Tcomplex}.

\noindent\textit{Step 1:
Let $g: \C^m \to \C$ be defined by setting $g( \bz ) := f( \bz ) - f( \bzero )$
for all $\bz \in \C$. The function $g$ has a power-series representation
in $\bz$ and} $\overline{\bz}$ with non-negative coefficients as seen
in~\tup{(\ref{Efmp})}.

First, we note the complex Hermitian version of Lemma~\ref{LWeyl} implies
$g[-] : \cS_N^{( \bk )}(\C) \to \overline{\cS_N^{(l+1)}}(\C)$ for all
$N \geq k_{\max}$. For any $t_0 \in ( 0, \infty )$ and
$A_1$, \ldots, $A_m \in \cS_n^{(0)}( \C )$, the block-diagonal matrix
\[
B^{(p)} := %
( -\eps \Id_{k_p} ) \oplus %
\bzero_{( k_{\max} - k_p ) \times ( k_{\max} - k_p )} \oplus A_p^{\oplus ( l + 2 )} %
\in \cS_N^{(k_p)}( \C )
\]
for all $p \in [ 1 : m ]$, where $N = k_{\max} + ( l + 2 ) n$.
Our initial observation implies that the
block-diagonal matrix $g[\bA]^{\oplus ( l + 2 )}$ can have at most $l + 1$
negative eigenvalues, but this is only possible if $g[\bA]$ is positive
semidefinite. We are now done, by Theorem~\ref{Tcomplex}.

\noindent\textit{Step 2:
The preserver $f$ satisfies \tup{(3)(a)} and \tup{(3)(b)}.}

This is obtained by restricting $f$ to act on $\R^m$,
in which case Theorem~\ref{Tmain} gives the representation
\[
f( \bx ) = f( \bzero ) + %
\sum_{\balpha \in \Z_+^{m_0} \setminus \{ \bzero \}} c_\balpha \bx_{[1 : m_0 ]}^\balpha + %
\sum_{p = m_0 + 1}^m c'_p x_p \qquad \text{for all } \bx \in \R^m,
\]
where each coefficient $c_\balpha$ and $c_p'$ is non-negative.
On the other hand, by the previous step,
\[
f( \bx ) = f( \bzero ) + \sum_{\balpha' \in \Z_+^m \setminus \{ \bzero \}}
f_{\balpha'} \bx^{\balpha'} \quad \text{ for all } \bx \in \R^m, \quad
\text{where } f_{\balpha'} := \sum_{\balpha + \bbeta = \balpha'}
c_{\balpha, \bbeta} \geq 0;
\]
(Note that both power series are absolutely convergent on $\C^m$.)
Equating the two forms of $f$ and applying the identity theorem for
several real variables, it follows that every monomial in the series
representation for $f$ other than those occurring in~(\ref{Epresrep3})
must vanish. This proves the claim.

\noindent\textit{Step 3:
The preserver $f$ satisfies \tup{(3)(c)}.}

We let $P_+ := \{ p \in [ m_0 + 1 : m ] : c_p > 0 \}$ and
$\overline{P}_+ := \{ p \in [ m_0 + 1 : m ] : d_p > 0 \}$.
If the set $P_+ \cup \overline{P}_+$ is empty then we take
\[
B^{(p)} := \left\{\begin{array}{ll}
\bzero_{k_{\max} \times k_{\max}} & \text{if }p \in [ 1 : m_0 ],\\[1ex]
-\Id_{k_p \times k_p} \oplus \bzero_{( k_{\max} - k_p ) \times ( k_{\max} - k_p )} & %
\text{if } p \in [ m_0 + 1 : m ].
\end{array}\right.
\]
By assumption, we have that
$f[ \bB ] = F( \bzero) \bone_{k_{\max} \times k_{\max}}$
is Hermitian and has at most $l$ negative eigenvalues.
This shows (3)(c).

Henceforth, we suppose $P_+ \cup \overline{P}_+$ is non-empty.
To see that  $F( \bzero_{m_0} ) = f( \bzero_m )$ is real, we let
\[
B^{(p)} := \bzero_{( k_{\max} + 1 - k_p ) \times ( k_{\max} + 1 - k_p )} \oplus %
 ( -\Id_{k_p} ) \qquad \text{if } p \in [ 1 : m ].
\]
Then $f[ \bB ]$ is Hermitian, by assumption, so its
$( 1, 1 )$ entry $f( \bzero )$ is real.

With Lemma~\ref{L2x2} at hand, we now proceed as in the final steps in
the  proof of Theorem~\ref{Tmulti2}. We let
\begin{equation}
K' := \sum_{p \in P_+ \cap \overline{P}_+} k_p \quad \text{and} \quad %
K'' := \sum_{p \in P_+ \symdiff \overline{P}_+} k_p,
\end{equation}
where the symmetric difference
$P_+ \symdiff \overline{P}_+ := %
( P_+ \cup \overline{P}_+ ) \setminus ( P_+ \cap \overline{P}_+ )$.
We need to show that $\bone_{F( \bzero ) < 0} + 2 K' + K'' \leq l$.

If $P_+ \symdiff \overline{P}_+ = \emptyset$, we skip this paragraph.
If not, we let $\{ v_1 = {\bf 1}_{K'' + 1}, v_2, \dots, v_{K'' + 1} \}$ be an
orthogonal basis of $\R^{K'' + 1}$ and extend the vectors
$v_2$, \ldots, $v_{K'' + 1}$ to lie in $\R^{K'' + 2 K' + 1}$
by padding with zeros: we set
$\tilde{v}_j := v_j \oplus \bzero_{2 K'}$ for $j \in [ 2 : K'' + 1 ]$.
We fix a partition
$\{ J_p : p \in P_+ \symdiff \overline{P}_+ \}$ of $[ 2 : K'' + 1 ]$
such that $| J_p | = k_p$ for all $p$.
Given $\delta \in ( 0, \infty )$, we let
\[
A^{(p)} := %
\frac{\delta}{\epsilon_p | P_+ \symdiff \overline{P}_+ |} %
\bone_{K'' + 2 K' + 1} \bone_{K'' + 2 K' + 1}^T - %
\frac{1}{\epsilon_p} \sum_{j \in J_p} \tilde{v}_j \tilde{v}_j^T
\]
for all $p \in P_+ \symdiff \overline{P}_+$,
where $\epsilon_p := c_p$ if $p \in P_+ \setminus \overline{P}_+$ and
$\epsilon_p := d_p$ if $p \in \overline{P}_+ \setminus P_+$.
Note that $A^{(p)} \in \cS_{K'' + 2 K' + 1}^{(k_p)}( \C )$
for all $p \in P_+ \symdiff \overline{P}_+$, by Lemma~\ref{Lsylv}.

Taking $N := K'' + 2 K' + k_{\max}$, we use the inflation operator
$\up$ from Definition~\ref{Dinflation} to obtain the test matrix
\[
B^{(p)} := \left\{\begin{array}{ll}
\bzero_{N \times N} & \text{if } p \in [ 1 : m_0 ],\\[1ex]
\up( A^{(p)} ) & \text{if } p \in P_+ \Delta \overline{P}_+,\\[1ex]
\bzero_{( N - k_p ) \times ( N - k_p )} \oplus ( -\Id_{k_p} ) & %
\text{if } p \in [ m_0 + 1 : m ] \setminus ( P_+ \cup \overline{P}_+ ),
\end{array}\right.
\]
where
$\pi := \bigl\{ \{ 1 \}, \ldots, \{ K'' + 2 K' \}, \{ K'' + 2K' + 1, \ldots, N \} \bigr\}$.

It remains to define the test matrix $B^{(p)}$ for
$p \in P_+ \cap \overline{P}_+ = \{ p_1 < \cdots < p_r \}$.
For each $s \in [ 1 : r ]$, we let
\[
k_{s)} := \sum_{t = 1}^{s - 1} k_{p_t} \qquad \text{and} \qquad %
k_{(s} := \sum_{t = s + 1}^r k_{p_t},
\]
and take
\[
B^{(p_s)} := %
\up\bigl( \bzero_{( K'' + 1 + 2 k_{s)} ) \times ( K'' + 1 + 2 k_{s)} )} \oplus %
( -A )^{\oplus k_{p_s}} \oplus %
\bzero_{2 k_{(s} \times 2 k_{(s}} \bigr),
\]
where $A$ is given by Lemma~\ref{L2x2}.

It is easy to see that the matrix $B^{(p)} \in \cS_N^{(k_p)}( \C )$ for
all $p \in [ 1 : m ]$, by Lemma~\ref{Linertia}, so
$f[ \bB ] \in \overline{\cS_N^{(l)}}( \C )$. Furthermore, a direct
computation shows that
\begin{align*}
f[ \bB ] & = F( \bzero ) \bone_{N \times N} + %
\sum_{p \in P_+ \setminus \overline{P}_+} c_p B^{(p)} + %
\sum_{p \in \overline{P}_+ \setminus P_+} d_p \overline{B^{(p)}} + %
\sum_{p \in P_+ \cap \overline{P}_+} %
\bigl( c_p B^{(p)} + d_p \overline{B^{(p)}} \bigr)\\[1ex]
 & = \up\bigl( (F( \bzero ) + \delta \bone_{K'' > 0} ) %
 \bone_{( K'' + 2 K' + 1 ) \times ( K'' + 2 K' + 1 )} \bigr)\\
 & \qquad + \up\biggl( \Bigl( -\bone_{K'' > 0} %
 \sum_{j = 2}^{K'' + 1} v_j v_j^T \Bigr) \oplus %
 \bigoplus_{s = 1}^r ( -c_{p_s} A - d_{p_s} \overline{A})^{\oplus k_{p_s}} \biggr),
\end{align*}
where $\bone_{K'' > 0} \sum_{j = 2}^{K'' + 1} v_j v_j^T$ equals
$0 \in \R^1$ if $K'' = 0$.
By Lemma~\ref{Linertia}, the matrix $f[ \bB ]$ has
precisely as many negative eigenvalues as the matrix
\begin{multline*}
M := \bigl( F( \bzero ) + \delta \bone_{K'' > 0} \bigr) %
\bone_{( K'' + 2 K' + 1 ) \times ( K'' + 2 K' + 1) }\\
+ \biggl( \Bigl( -\bone_{K'' > 0} %
\sum_{j = 2}^{K'' + 1} v_j v_j^T \Bigr) \oplus %
 \bigoplus_{s = 1}^r ( -c_{p_s} A - d_{p_s} \overline{A})^{\oplus k_{p_s}} \biggr).
\end{multline*}
Lemma~\ref{Lsylv} and Lemma~\ref{L2x2} imply that the second term in this sum
has exactly $K'' + 2 K'$ negative eigenvalues.
Moreover, the vector $\bone_{K'' + 2 K' + 1}$ is not in the
column space of this second term, since it is orthogonal to each
vector $\tilde{v}_j$ and the column spaces of the remaining block factors
\[
\bzero_{( K'' + 1 + 2 k_{s)} ) \times ( K'' + 1 + 2 k_{s)} )} \oplus %
( -c_{p_s} A - d_{p_s} \overline{A} )^{\oplus k_{p_s}} \oplus %
\bzero_{2 k_{(s} \times 2 k_{(s}}
\]
together span $\bzero_{K'' + 1} \oplus \C^{2 K'}$.

To conclude the proof, we consider two cases.
If $F( \bzero ) \geq 0$ then Lemma~\ref{Lsylv} gives that
$M$ has exactly $K'' + 2 K'$ negative eigenvalues, so
$K'' + 2 K' \leq l$, as desired. If, instead, we have that
$F( \bzero ) < 0$ then we can ensure that
$F( \bzero ) + \delta \bone_{K'' > 0} < 0$, shrinking $\delta$ if
necessary, and then $M$ has exactly $K'' + 2 K' + 1$ negative
eigenvalues. Again, this is at most $l$, which concludes the proof.
\end{proof}

We conclude by adapting Corollaries~\ref{Cpreserver} and~\ref{Cinertia}
to the complex setting. As conjugate maps are involved, we re-state both
results.

\begin{corollary}
Let $k$ and $m$ be positive integers and let $f : I^m \to \C$.
\begin{enumerate}
\item The entrywise transform $f[ - ]$ sends $\cS_n^{(k \bone_m^T)}( \C )$
to $\cS_n^{(k)}$ for all $n \geq k$ if and only if $f( \bz ) \equiv c z_{p_0}$
or $f( \bz ) \equiv c \overline{z_{p_0}}$ for a constant $c > 0$ and some
$p_0 \in [ 1 : m ]$, or, when $k = 1$, we may also have $f( \bz ) \equiv -c$ for some
$c > 0$.

\item The entrywise transform $f[ - ]$ sends
$\overline{\cS_n^{(k \bone_m^T)}}( \C )$ to
$\overline{\cS_n^{(k)}}$ for all $n \geq k$
if and only if $f( \bz ) \equiv c z_{p_0} + d$ or $f( \bz ) \equiv c \overline{z_{p_0}} + d$
for some $p_0 \in [ 1: m ]$, with either $c = 0$ and $d \in \R$, or $c > 0$ and $d \geq 0$.

\item If $m = 1$ and $k' \in \{ 0, k \}$, the entrywise transform
$f[ - ]$ preserves the inertia of all matrices in $\cS^{(k')}( \C )$ if and only if
$f( z ) \equiv c z$  or $f( z ) \equiv c \overline{z}$ for some $c > 0$.
\end{enumerate}
\end{corollary}

\begin{proof}
The proof of the first two parts are identical to the proof of the $I = ( -\rho,\rho )$
case of Corollary~\ref{Cpreserver}, except that we use
Theorem~\ref{Tmain-complex} instead of Theorem~\ref{Tmulti2}.
To prove~(3), the reverse inclusion is immediate; conversely, 
Theorem~\ref{Tmain-complex} gives that either $f( z ) \equiv c z + d$ or
$f( z ) \equiv c \overline{z} + d$, with $c \geq 0$ and with
$d \geq 0$ if $c > 0$. Now restricting to matrices
with real entries, we are done by the corresponding implication in Theorem~\ref{Tinertia}.
\end{proof}

\appendix

\section{Absolutely monotone functions of several variables}\label{Appendix}

The purpose of this Appendix is to provide the proof of a result on absolutely
monotone functions of several variables that seems to not be readily available
in the literature, but was used in previous work~\cite{BGKP-hankel}
and is relevant to the present paper.

Given an open interval $I \subset \R$ and a positive integer $m$,
let $\bx = ( x_1, \ldots, x_m ) \in I^m$ and
$\balpha = ( \alpha_1, \ldots, \alpha_m ) \in \Z_+^m$. Set
$| \balpha | := \alpha_1 + \cdots + \alpha_m$.

Recall that a smooth function $f : I^m \to \R$ is \emph{absolutely monotone} if
\[
\partial^\balpha f( \bx ) = %
\frac{\partial^{| \balpha |} f}%
{\partial x_1^{\alpha_1} \cdots \partial x_m^{\alpha_m}}%
( x_1, \ldots, x_m )
\geq 0 \quad \textrm{for all } \balpha  \in \Z_+^m \textrm{ and } \bx \in I^m.
\]
If, instead, the function $f$ is such that
\[
(-1)^{| \balpha |} \partial^\balpha f( \bx ) \geq 0 %
\quad \textrm{for all } \balpha  \in \Z_+^m \textrm{ and } \bx \in I^m
\]
then $f$ is said to be \emph{completely monotone}.

A function $f : [ 0, \rho )^m \to \R$, where $0 < \rho \leq \infty$,
is said to be \emph{absolutely monotone} if the restriction of~$f$
to~$( 0, \rho )^m$ is absolutely monotone, as defined above,
and $f$ is continuous on $[ 0, \rho )^m$.
Step~II in the proof of \cite[Theorem 5]{Ressel} shows that
such a function has non-negative one-sided derivatives at
the boundary points of its domain.

Above and in previous work~\cite{BGKP-hankel}, we use
the fact that absolutely monotone functions 
have power-series representations with non-negative Maclaurin coefficients.
This result is used for functions with domains of the form $( 0, \rho )^m$, 
where $0 < \rho \leq \infty$ and $m \geq 1$. For $m = 1$, this is a special case
of Bernstein's theorem \cite{Bernstein}, and for $m = 2$ the power-series representation is derived in Schoenberg's paper \cite{Schoenberg33}
using completely monotone functions.

However, we were unable to find in the literature a reference for the case $m > 2$.
When $I = [ 0, \rho )$ and $\rho$ is finite, this representation theorem, for all
$m \geq 1$, is found in Ressel's work \cite[Theorem~8]{Ressel};
the extension to the case $\rho = \infty$ is immediate,
by the identity theorem.

To fill this gap, we state the following theorem and provide its proof.

\begin{theorem}\label{Tabsmon}
Let $I = ( 0, \rho )$, where $0 < \rho \leq \infty$,
and let $m$ be a positive integer.
The smooth function $f : I^m \to \R$ is absolutely monotone if
and only if $f$ is represented on~$I^m$ by a power series with
non-negative Maclaurin coefficients:
\[
f( \bx ) = \sum_{\balpha \in \Z_+^m} c_\balpha \bx^\balpha %
\qquad \textrm{for all } \bx \in I^m, %
\textrm{ where } c_\balpha \geq 0 \textrm{ for all } \balpha.
\]
\end{theorem}
Note that any function with such a representation extends to a
real-analytic function on the domain $( -\rho, \rho )^m$.

Theorem~\ref{Tabsmon} was implicitly used to prove the
multi-variable version of Schoenberg's Theorem~\ref{T000}
with $I = ( 0, \rho )^m$
in our previous work \cite{BGKP-hankel}, where it is used in turn
to show this theorem for $I = [ 0, \rho )^m$ and
$I = ( -\rho, \rho )^m$. The same result is used in the present work,
to prove Proposition~\ref{Pabsmon}(1) for all three choices of $I$,
leading to the characterization results of Theorems~\ref{Tto0}
and~\ref{Tmulti2} and their corollaries.

We now turn to the proof of Theorem~\ref{Tabsmon}. While it is likely
that it would be possible to show this result for $m > 2$ by adapting the proof
for the $m=2$ case given by Schoenberg~\cite{Schoenberg33},
we proceed differently here: we use Ressel's theorem and a natural group
operation on convex cones, well known in Hardy-space theory.

\begin{proof}[Proof of Theorem~\ref{Tabsmon}]
The reverse implication is clear. To prove the other,
we note that, from the previous remarks, we need only to show that
an absolutely monotone function on $( 0, \rho )^m$
has a continuous extension to $[ 0, \rho )^m$ when $\rho$ is finite.
Moreover, by scaling, we may assume that $\rho = 1$.
 We offer two different paths to show that such an extension exists.

\noindent\textit{Path 1:} We note first that
\[
g : ( 0, \infty )^m \to [ 0, \infty); \ %
\bx \mapsto f( e^{-x_1}, \ldots, e^{-x_m} )
\]
is completely monotone, because an inductive argument shows that
\[
( \partial^\balpha g )( \bx ) = ( -1 )^{| \balpha |} g( \bx ) %
p_\balpha( e^{-x_1}, \ldots, e^{-x_m} ) %
\qquad \textrm{for all } \balpha \in \Z_+^m,
\]
where $p_\balpha$ is a polynomial with non-negative coefficients.
Hence, by \cite[Corollaries~2.1 and~2.2]{HS},
the function $g$ is real analytic on $( 0, \infty )^m$.
Composition with the change of coordinates
\[
T : ( 0, 1 )^m \to ( 0, \infty )^m; \ %
( x_1, \ldots, x_m ) \mapsto ( -\log x_1, \ldots, -\log x_m ) 
\]
now shows that $f = g \circ T$ is real analytic on $( 0, 1 )^m$.

Now fix $\epsilon \in ( 0, 1 / 2 )$ and consider the function.
\[
h : ( -\epsilon, 1 - \epsilon )^m \to \R; \ %
\bx \mapsto f( \bx + \epsilon \bone_m ).
\]
It is immediate that $h$ is absolutely monotone on $[ 0, 1 - \epsilon)^m$,
so it is represented there by a power series with non-negative Maclaurin
coefficients, by \cite[Theorem~8]{Ressel}:
\begin{equation}\label{Eseries}
f( \bx + \epsilon \bone_m) = h( \bx ) = %
\sum_{\balpha \in \Z_+^m} d_\balpha \bx^\balpha
\qquad \textrm{for all } \bx \in [ 0, 1 - \epsilon )^m,
\end{equation}
where $d_\balpha \geq 0$ for all $\balpha \in \Z_+^m$. Moreover,
this series is absolutely convergent for all
$\bx \in [ -\epsilon, 1 - \epsilon )^m$, since $\epsilon < 1/2$.
Thus $h$ admits a continuous extension to the boundary.
Since $f$ is real analytic and agrees with the power
series~(\ref{Eseries}) on an open set, it admits a continuous
extension to $[ 0, 1 )^m$, as desired.\qed

\noindent\textit{Path 2:} The following approach was
communicated to us by Paul Ressel. Given any
point $\bx_0 \in [ 0, 1 )^m$, there exists $\epsilon_0 > 0$
such that $\bx_0 + \epsilon_0 \bone_m \in [ 0, 1 )^m$. If
\[
g_{\bx_0} : ( 0, \epsilon_0 ] \to \R; \ %
t \mapsto f(\bx_0 + t \bone_m )
\]
then $g_{\bx_0}'' \geq 0$ and so $g_{\bx_0}$ is convex on
$( 0, \epsilon_0 ]$. Hence $g_{\bx_0}( 0^+ )$ exists, is finite
and agrees with $g_{\bx_0}( 0 ) = f( \bx_0 )$ if $\bx_0 \in ( 0, 1 )^m$.
We can now extend $f$ to $[ 0, 1 )^m$ by setting
\[
\wt{f} : [ 0, 1 )^m \to \R; \ \bx_0 \mapsto g_{\bx_0}( 0^+ ).
\]
One can verify that this function satisfies the ``forward difference''
definition of absolute monotonicity, as given on
\cite[pp.~259--260]{Ressel}. Hence, by \cite[Theorem 8]{Ressel},
$\wt{f}$ is represented by a convergent series with non-negative
Maclaurin coefficients, and therefore so is~$f$.
\end{proof}

\subsection*{Acknowledgements}

The authors hereby express their gratitude to the University of
Delaware for its hospitality and stimulating working environment, and in
particular to the Virden Center in Lewes, where this work was
initiated.
The authors also thank the Institute for Advanced Study,
Princeton, and the American Institute of Mathematics, Pasadena, for their
hospitality while this work was continued.
A.K.~thanks Paul Ressel for useful discussions related to
Theorem~\ref{Tabsmon}. We also thank the referee for helpful suggestions.

A.B.~was partially supported by Lancaster University while some of this
work was undertaken, and is grateful for the support of the SPARC
project \textit{Noncommutative Analysis and Probability},
funded by the MHRD, Government of India,
and the hospitality of the Indian Institute of Science, Bangalore,
which led to the concluding portion of this work.

D.G.~was partially supported by a University of Delaware Research
Foundation grant, by a Simons Foundation collaboration grant for
mathematicians, by a University of Delaware strategic initiative
grant, and by NSF award \#2350067.

A.K.~was partially supported by the Ramanujan Fellowship
SB/S2/RJN-121/2017 and SwarnaJayanti Fellowship grants SB/SJF/2019-20/14
and DST/SJF/MS/2019/3 from SERB and DST (Govt.~of India),
a Shanti Swarup Bhatnagar Award from CSIR (Govt.\ of India), and
the DST FIST program 2021 [TPN--700661]. He is also grateful for
the support (via IISc) of the aforementioned SPARC project
\textit{Noncommutative Analysis and Probability},
and the hospitality of the University of Plymouth, UK,
where this work was concluded.

M.P.~was partially supported by a Simons Foundation collaboration
grant.



\end{document}